\documentclass[10pt]{amsart}
\usepackage{amsmath,graphicx,amssymb,amsthm}
\usepackage{color}
\def\A{\mathcal{A}}
\def\B{\mathcal B}

\def\W{\mathcal W}
\def\D{\mathcal D}
\def\B{\mathcal B}
\def\C{\mathcal C}
\def\H{\mathcal H}

\def\L{\mathcal L}
\def\M{\mathcal{M}}

\def\P{\mathcal P}

\def\Z{\mathcal Z}

\def\amslatex{$\mathcal{A}\kern-.1667em\lower.5ex\hbox{$\mathcal{M}$}\kern-.125em\mathcal{S}$-\LaTeX}

\newtheorem{set}{set}[section]
\newtheorem{Corollary}[set]{Corollary}
\newtheorem{Definition}[set]{Definition}

\newtheorem{Lemma}[set]{Lemma}

\newtheorem{Proposition}[set]{Proposition}

\newtheorem{Theorem}[set]{Theorem}
\newcommand{\define}{\mathrel{\hbox{$\equiv$\hskip -.90em \lower .47ex \hbox{$\leftharpoondown$}}}}
\newcommand{\enifed}{\mathrel{\hbox{$\equiv$\hskip -.90em \lower .47ex \hbox{$\rightharpoondown$}}}}

\begin{document}
\title {\bf   $R$-diagonal and $\eta$-diagonal  Pairs of Random Variables}
\author{Mingchu Gao}
\address{School of Mathematics and Information Science,
Baoji University of Arts and Sciences,
Baoji, Shaanxi 721013,
China; and
Department of Mathematics,
Louisiana College,
Pineville, LA 71359, USA}
\email{mingchu.gao@lacollege.edu}

\date{}
\begin{abstract} This paper is devoted to studying  $R$-diagonal and $\eta$-diagonal pairs of random variables.    We generalize circular elements to the bi-free setting,
defining bi-circular element pairs of  random variables,  which provide examples of $R$-diagonal pairs of random variables.  Formulae are given for calculating the distributions of
 the product pairs of two $*$-bi-free $R$-diagonal pairs. When focusing on pairs of left acting operators and right acting operators from finite von Neumann algebras in the standard
 form, we characterize  $R$-diagonal pairs in terms of the $*$-moments of the random variables, and of  distributional invariance of the random variables under multiplication by free unitaries.
 We define $\eta$-diagonal pairs of random variables, and  give a characterization of  $\eta$-diagonal pairs in terms of the $*$-distributions of the random variables.
 If every non-zero element in a $*$-probability space has a non-zero $*$-distribution,  we  prove that  the unital algebra generated by a  $2\times 2$ off-diagonal matrix with entries of a  non-zero random variable $x$ and its adjoint $x^*$ in the algebra and the diagonal $2\times 2$ scalar matrices can never be Boolean independent from
  the  $2\times 2$ scalar matrix algebra with amalgamation over the diagonal scalar matrix algebra.
\end{abstract}
\maketitle
{\bf AMS Mathematics Subject Classification (2010)} 46L54.

{\bf Key words and phrases} $R$-diagonal pairs of random variables, $\eta$-diagonal pairs of random variables, bi-Boolean independence, bi-free independence, Boolean independence with amalgamation.
\section{Introduction}
We consider the framework of a $*$-probability space $(\A,\varphi)$, that is, $\A$ is unital $*$-algebra, $\varphi$ is a linear functional on $\A$ such that $\varphi(1_\A)=1$ and $\varphi(a^*)=\overline{\varphi(a)}$, where $1_\A$ is the unit of algebra $\A$, and $a\in \A$.   For an element $a\in \A$, the family of expectations of the words formed by  $a$ and $a^*$, $\{\varphi(a^{\omega(1)}\cdots a^{\omega(n)}): \omega:\{1,..., n\}\rightarrow \{1,*\}, n=1,...\}$, which is called the family of $*$-moments of $a$, carries significant probabilistic, algebraic, and analytic information for $a$. For instance, when $\A$ is a $C^*$-algebra, and $\varphi$ is faithful and positive, the family determines completely the unital $C^*$-subalgebra generated by $a$; a similar fact is true in the framework of von Neumann algebras (see e.g. Remark 1.8 in \cite{DV1}). In the  $C^*$-probability space $(B(\H), \varphi_\xi)$ (i.e., $\H$ is a Hilbert space, $B(\H)$ is the $C^*$-algebra of all bounded linear operators on $\H$, and $\varphi_\xi(T)=\langle T\xi, \xi\rangle$, for $T\in B(\H)$ and a fixed unit vector $\xi\in \H$), a normal operator can be transformed by the spectral theorem to a multiplication operator  on a $L^2$-space. Therefore, the distribution of such an operator is a probability measure on a compact subset of the complex plane. From a probabilistic point of view, the phenomenon of normal operators takes place in the classical commutative world with plenty of commutative tools such as classical probability and classical analysis. On the other hand, non-normal operators live in a truly non-commutative realm. The $*$-distribution of a non-normal operator is merely a unital linear functional on the polynomial algebra, and is much harder to analyze. It is therefore quite important to have sufficiently large classes of non-normal operators which can be treated probabilistically.

 $R$-diagonal elements are among the most prominent  non-normal operators arising from free probability.   The concept of $R$-diagonal elements in the tracial case was introduced in  \cite{NS1}, as a generalization of Haar unitaries and circular elements,  and was subsequently found to play an important rule in several problems in free probability (see e.g. \cite{NSS1}, \cite{NSS2}, and \cite{NS2}). The $R$-diagonal elements in the general (not necessarily tracial) case were treated in \cite{KS} and \cite{NSS}. The class of $R$-diagonal elements  has received quite a bit of attention in the free probability literature. In particular, elements with $R$-diagonal distributions were among the first examples of non-normal elements in a $W^*$-probability space for which the Brown spectral measure was calculated explicitly in \cite{HL}, and for which the Brown measure techniques could be used to find invariant subspaces in \cite{SS}. $R$-diagonal $*$-distributions also appear in large $N$ limit results for random matrices, in connection to the single ring theorem (\cite{GKZ}).

  An element $a$ in a $*$-probability space $(\A, \varphi)$ is said to be {\sl $R$-diagonal} if  the free cumulant $$\kappa_n(a_1, ..., a_n)=0,$$ unless the arguments $a_1, ..., a_n\in \{a, a^*\}$ appear alternatingly and $n$ is even (Definition 15.3 in \cite{NS}). Therefore, the distribution of a $R$-diagonal element is determined by two sequences $$\{\alpha_n:=\kappa_{2n}(a, a^*, ..., a, a^*): n=1, 2,...\},\{\beta_n:=\kappa_{2n}(a^*, a, ..., a^*, a): n=1, 2, ...\},$$ which are called determining sequences of $a$.
 Nica, Shlyakhtenko, and Speicher characterized $R$-diagonal elements  in terms of their $*$-moments, the invariance of their distributions under multiplication by free unitaries, and the freeness of the corresponding matrix from the $2\times 2$ scalar matrix algebra with amalgamation over the diagonal scalar matrix algebra (Theorem 1.2 in \cite{NSS}). Krawczyk and Speicher proved that $ab$ is $R$-diagonal if  $a$ is $R$-diagonal, and $a$ and $b$ are $*$-free (Proposition 3.6 in \cite{KS}). Moreover, if $b$ is also $R$-diagonal, the authors gave the formulae to compute the determining sequences of $ab$ in terms of the determining sequences of  $a$ and $b$ (Proposition 3.9 in \cite{KS}).

 Skoufranis introduced the concept of $R$-diagonal pairs of random variables, as an example and a resource to produce $R$-cyclic pairs of matrices of random variables in \cite{PS} (Example 4.7 in \cite{PS}; see also Definition 2.1 in this paper). Skoufranis proved that a two faced pair of left acting matrices of random variables  and right acting matrices of random variables is bi-free from the pair of the left acting scalar matrix algebra, and the right acting scalar matrix algebra with amalgamation over the diagonal scalar matrix algebra $\D_2$ if and only if the two faced pair of matrices of random variables is $R$-cyclic (Theorem 4.9 in \cite{PS}), which implies that $(x,y)$ is $R$-diagonal if and only if $(L(X), R(Y))$ is bi-free from $(L(M_2(\mathbb{C})), R(M_2(\mathbb{C})^{op}))$ with amalgamation over $\D_2$ with respect to $F_2:M_2(\A)\rightarrow \D_2$, $F_2([a_{ij}])=diag(\varphi(a_{11}), \varphi(a_{22}))$, for $[a_{ij}]\in M_2(\A)$,  where $X=\left(\begin{matrix}0&x\\
 x^*&0\end{matrix}\right)$ and $Y=\left(\begin{matrix}0&y\\
 y^*&0\end{matrix}\right)$, and $diag(\varphi(a_{11}), \varphi(a_{22}))$ is the diagonal matrix in $\D_2$ with $(1,1)$-entry $\varphi(a_{11})$ and $(2,2)$-entry $\varphi(a_{22})$ (Proposition 2.21 in \cite{GK}). Hence, Skoufranis' work in Section 4 of \cite{PS} implies a bi-free analogue of the characterization of  $R$-diagonal pairs in terms of freeness with amalgamation (Condition 5 in Theorem 1.2 in \cite{NSS}). Let $(x_1, y_1)$ and $(x_2, y_2)$ be $*$-bi-free pairs of random variables in a $*$-probability space $(\A,\varphi)$.  G. Katsimpas \cite{GK} proved that if $(x_1, y_1)$ is $R$-diagonal, the $(x_1x_2, y_2y_1)$ and $(x_1^n, y_1^n)$ are $R$-diagonal. If, furthermore, $(x_2, y_2)$ is also $R$-diagonal, then $(x_1x_2, y_1y_2)$ is $R$-diagonal, too (Theorems 3.2, 3.5, and Proposition 3.4 in \cite{GK}). G. Katsimpas \cite{GK} also proved  distributional invariance of a $R$-diagonal pair of random variables under  multiplication by a $*$-bi-free bi-Haar unitary pair.

 In this paper, we continue the study on $R$-diagonal pairs of random variables. Haar unitaries and circular elements are two typical examples of $R$-diagonal random variables. Katsimpas  proved that bi-Haar unitary pairs are $R$-diagonal (Corollary 2.18 in \cite{GK}). We generalize circular elements to the bi-free setting, defining bi-circular element
 pairs, and prove that such a pair is $R$-diagonal. We provide  formulae for calculating determining sequences for the product pairs of two $*$-bi-free  $R$-diagonal pairs of random variables. According to Voiculescu's philosophy on bi-free probability (\cite{DV}), it is natural and more meaningful to investigate bi-free probabilistic phenomena in the framework  of pairs of left acting operators and right acting operators. We thus focus  on the study  of $R$-diagonal pairs of left acting operators and right acting operators from  finite von Neumann algebras in the standard form (see Section 6 in \cite{DV} for the construction). In this case, we characterize  $R$-diagonal pairs in terms of the $*$-moments of the random variables, and of distributional invariance  of the random variables under multiplication by free unitaries, generalizing the main work in \cite{NSS} to the bi-free setting.

 From a combinatorial point of view, the main difference between a variety of (non-commutative) probability theories consists of choosing different partitions in defining cumulants.   Let $\P(n)$ be the set of all partitions of the set $\{1,..., n\}$, $NC(n)$ the set of all non-crossing partitions, and $IN(n)$ the set of all interval partitions (i.e., each block of the partition is an interval $\{p+1,..., q\}\subseteq \{1,..., n\}$ of natural numbers).     For $a_1, ..., a_n$ in a non-commutative probability space $(\A, \varphi)$, the classical cumulants were defined by $$\varphi(a_1\cdots a_n)=\sum_{\pi\in \P(n)}c_\pi(a_1,..., a_n),$$ where $(c_{\pi})_{\pi\in\mathcal{P}(n), n\in \mathbb{N}}$ is the family of cumulants, which is a multiplicative family of functions on $\P:=\coprod_{n\in \mathbb{N}}\P(n)$. Unital subalgebras $\A_1$ and $\A_2$ are independent (or, called tensorially  independent) in $(\A, \varphi)$ if and only if all mixed cumulants of elements from $\A_1$ and $\A_2$ vanish (Theorem 11.32 in \cite{NS}). When  restricting the partitions to non-crossing ones, we get free cumulants:
 $$\varphi(a_1\cdots a_n)=\sum_{\pi\in NC(n)}\kappa_{\pi}(a_1,..., a_n),$$ where  $(\kappa_{\pi})_{\pi\in NC(n), n\in \mathbb{N}}$ is called the family of free cumulants, a multiplicative family of functions on $\coprod_{n\in \mathbb{N}}NC(n)$. Unital subalgebras $\A_1$ and $\A_1$ are freely independent in $(\A, \varphi)$ if and only if all mixed free cumulants of elements from $\A_1$ and $\A_2$ vanish (Theorem 11.16 in \cite{NS}). Furthermore, when summing only interval partitions, we get
$$\varphi(a_1\cdots a_n)=\sum_{\pi\in IN(n)}B_\pi(a_1,..., a_n),$$ where $(B_{\pi})_{ \pi\in IN(n), n\in \mathbb{N}}$ is the family of Boolean cumulants, a multiplicative family of functions on $\coprod_{n\in \mathbb{N}} IN(n)$. Non-unital subalgebras $\A_1$ and $\A_1$ are Boolean independent in $(\A,\varphi)$ if and only if all mixed Boolean cumulants of elements from $\A_1$ and $\A_2$ vanish (\cite{GS} and \cite{SW}).

The free cumulants of a random variable $a$ can be used to define a formal series, called $R$-transform (or $R$-series) of $a$, $R_{a}(z)=\sum_{n=1}^\infty \kappa_n(a)z^n,$ where $\kappa_n=\kappa_{1_n}$ and $1_n=\{\{1,..., n\}\}$ is the one-block partition of the set $\{1,..., n\}$. Similarly, The Boolean cumulants of $a$ can be used to define $\eta$-series $\eta_a(z)=\sum_{n=1}^\infty B_n(a)z^n$, where $B_n=B_{1_n}$. With the same spirit, Gu and Skoufranis \cite{GS} defined bi-Boolean cumulants, bi-Boolean independence, and  bi-Boolean $\eta$-series.
A $R$-diagonal element has a `diagonal' $R$-series $$R_{a, a^*}(z, z^*)=\sum_{n=1}^\infty \kappa_{2n}(a, a^*, ..., a, a^*)(zz^*)^n+\sum_{n=1}^\infty \kappa_{2n}(a^*, a, ..., a^*, a)(z^*z)^n.$$ Thus, Bercovici et al. \cite{BNNS} call an element $a\in A$ {\sl $\eta$-diagonal} if its $\eta$-series is `diagonal' $$\eta_a(z)=\sum_{n=1}^nB_{2n}(a, a^*, ..., a, a^*)(zz^*)^n+\sum_{n=1}^nB_{2n}(a^*, a, ..., a^*, a)(z^*z)^n.$$ The authors of \cite{BNNS} gave a characterization of a $\eta$-diagonal element in terms of the $*$-moments of the element (Theorem 2.8 in \cite{BNNS}).

In this paper, we define {\sl $\eta$-diagonal} pairs of random variables and give a characterization of a $\eta$-diagonal pair in terms of the $*$-moments of the random variables,
generalizing the work in Section 2 of \cite{BNNS} to the bi-Boolean case. The property of being $R$-diagonal  for a random variable can be characterized in terms of the
 freeness of the associated matrix of the random variable from the  scalar $2\times 2$ matrix algebra with amalgamation  over
 the diagonal scalar matrix algebra (\cite{NSS}). It is natural and interesting to study a similar question in the $\eta$-diagonal case. We find that if every non-zero element in a $*$-probability space $(\A,\varphi)$ has a non-zero $*$-distribution, then, for a non-zero $x\in \A$,  the unital subalgebra $\Z$ generated by the
 matrix $\left(\begin{matrix}0&x\\
x^*&0\end{matrix}\right)$ and  diagonal $2\times 2$ scalar matrices can never be   Boolean independent from the scalar matrix algebra $M_2(\mathbb{C})$ with amalgamation over the diagonal scalar matrix algebra $\D_2$.

Besides this Introduction, this paper consists of four sections. In Section 2,   we define bi-circular element pairs of random variables, and prove that such a pair is $R$-diagonal
(Definition 2.3 and Theorem 2.4).  Formulae are given  to calculate the determining sequences of $(x_1x_2, y_2y_1)$ and $(x_1x_2, y_1y_2)$ for  $*$-bi-free pairs $(x_1, y_1)$ and $(x_2, y_2)$, if  both $(x_1,y_1)$ and $(x_2, y_2)$ are $R$-diagonal (Theorem 2.5 and Corollary 2.6).  In the single random variable case, it was proved that if $a$ is $R$-diagonal, then $aa^*$ and $a^*a$ are free (Corollary 15.11 in \cite{NS}). We prove that there is an $R$-diagonal pair of random variables $(x,y)$ in a $*$-probability space $(\A, \varphi)$ such that $(xx^*, yy^*)$ is not bi-free from $(x^*x, y^*y)$ (Theorem 2.7). In  Section 3 we aim to  study $R$-diagonal pairs of left acting operators  and right acting operators from finite von Neumann algebras in the standard form. We give characterizations of  $R$-diagonal  pairs in this case (Theorem 3.3). Section 4 is devoted to studying $\eta$-diagonal pairs of random variables. We characterize  $\eta$-diagonal pairs in terms of the $*$-moments of the random variables (Theorem 4.8). As in the $R$-diagonal case, we find an $\eta$-diagonal pair of random variables $(x,y)$, for which $(xx^*, yy^*)$ is not bi-Boolean independent from $(x^*x, y^*y)$ (Corollary 4.10).
  Finally, in Section 5, we study the Boolean independence of $\Z$ and $M_2(\mathbb{C})$ with amalgamation over the scalar diagonal matrix algebra (Theorem 5.2).

The reader is referred to \cite{NS} and \cite{VDN} for the basics on free probability, and to \cite{DV}, \cite{CNS1}, and \cite{CNS2} for the basics on bi-free probability.

{\bf Acknowledgement} The author would like to thank the anonymous referee(s) for carefully reading the original manuscript and pointing out tremendous typos and mistakes and giving suggestions to improve it.
\section{Products of bi-free $R$-diagonal pairs of random variables }

In this section, we study $R$-diagonal pairs of random variables, giving formulae to compute the distributions of the product pairs of two bi-free $R$-diagonal pairs of random variables.

Let $I$ and $J$ be two index sets, and  $\chi:\{1, 2, \cdots, n\}\rightarrow I\bigsqcup J$. We define a permutation $s_\chi$ of $\{1,2,..., n\}$ by $\chi^{-1}(I)=\{s_\chi(1)<s_\chi(2)<\cdots s_\chi(k)\}$ and $\chi^{-1}(J)=\{s_\chi(k+1)>s_\chi(k+2)>\cdots >s_\chi(n)\}.$ The permutation $s_\chi$ defines a new order on $\{1, 2, ..., n\}$: $s_\chi(1)\prec_{\chi} s_{\chi}(2)\prec_{\chi}\cdots \prec_{\chi} s_\chi(n)$.

Based on the ideas in defining $R$-diagonal random variables, Skoufranis \cite{PS} gave the following concept of $R$-diagonal pairs of random variables.
\begin{Definition}[Example 4.7 in \cite{PS}] Let $(\A,\varphi)$ be a $*$-probability space and $(x,y)$ be a pair of elements in $\A$. We say that $(x,y)$ is $R$-diagonal if all odd order bi-free cumulants of $((x,x^*), (y, y^*))$ are zero and $\kappa_\chi(z_1,..., z_{2n})=0$ unless the tuple $(z_{s_\chi(1)},..., z_{s_\chi(2n)})$ is one of the following forms
\begin{enumerate}
\item $(x^{\omega(1)},..., x^{\omega(k)}, y^{\omega(k+1)},..., y^{\omega(2n)})$, $\omega:\{1,..., 2n\}\rightarrow \{1,*\}$, $\omega(1)=1, \omega(i)\ne \omega(i+1), i=1, ..., 2n-1$, $0\le k\le 2n$,
\item $(x^{\omega(1)},..., x^{\omega(k)}, y^{\omega(k+1)},..., y^{\omega(2n)})$, $\omega:\{1,..., 2n\}\rightarrow \{1,*\}$, $\omega(1)=*, \omega(i)\ne \omega(i+1), i=1, ..., 2n-1$, $0\le k\le 2n$.
\end{enumerate}
\end{Definition}
The distribution of a $R$-diagonal pair of random variables is thus determined by the following sequences
$$\{\alpha_{\chi}=\kappa_{\chi}(z_1,..., z_{2n}): \chi: \{1,..., 2n\}\rightarrow \{l,r\}, n=1, ...\},$$ where $(z_{s_\chi(1)},..., z_{s_\chi(2n)})=(x^{\omega(1)},..., x^{\omega(k)}, y^{\omega(k+1)},..., y^{\omega(2n)})$, $\omega(1)=1$,    $\omega:\{1,..., 2n\}\rightarrow \{1,*\}$, and $\omega(i)\ne \omega(i+1)$, for $i=1, ..., 2n-1$, for  $V=\{i_1<\cdots <i_k\}\subseteq \{1,..., n\} $, $\alpha_{\chi}(V)=\kappa_{\chi|_{V}}(z_{i_1}, ..., z_{i_k})$;

$$\{\beta_{\chi}=\kappa_{\chi}(z_1,..., z_{2n}): \chi:\{1,..., 2n\}\rightarrow \{l,r\}, n=1,...\},$$ where $(z_{s_\chi(1)},..., z_{s_\chi(2n)})=(x^{\omega(1)},..., x^{\omega(k)}, y^{\omega(k+1)},..., y^{\omega(2n)})$, $\omega(1)=*$,  $\omega:\{1,..., 2n\}\rightarrow \{1,*\}$, $\omega(i)\ne \omega(i+1)$, for $i=1, ..., 2n-1$, for $V=\{i_1<\cdots <i_k\}\subseteq \{1,..., n\} $, $\beta_{\chi}(V)=\kappa_{\chi|_{V}}(z_{i_1}, ..., z_{i_k})$.

The two sequences $\{\alpha_{\chi}: \chi:\{1, ..., n\}\rightarrow \{l,r\}, n=1, ...\}$ and $\{\beta_{\chi}: \chi:\{1, ..., n\}\rightarrow \{l,r\}, n=1, ...\}$ are called the determining sequences of the $R$-diagonal pair $(x,y)$.

The bi-free generalization of Haar unitaries was first proposed in Definition 10.2.1 in \cite{CNS2} in the operator-valued setting. A scalar-valued version of the concept was given in \cite{GK}.
\begin{Definition}[Definition 2.15, \cite{GK}] A pair of unitaries $(u_l, u_r)$ in a $*$-probability space $(\A, \varphi)$ is a bi-Haar unitary pair if the algebras $alg(\{u_l, u_l^*\})$ and $alg(\{u_r, u_r^*\})$ commute and for $n, m\in \mathbb{Z}$, $$\varphi(u_l^nu_r^m)=\left\{\begin{matrix}1, & \text{if } m+n=0,\\
0,& \text{otherwise}.
\end{matrix}\right.$$
\end{Definition}
G. Katsimpas proved that a bi-Haar unitary pair is $R$-diagonal (Corollary 2.18 in \cite{GK}). Another typical example of $R$-diagonal random variables is the circular random variable (Lecture 15 in \cite{NS}). We  generalize circular elements to the bi-free setting, providing another kind of  examples of $R$-diagonal pairs of random variables.
\begin{Definition}Let $(z_{1,l}, z_{1,r})$ and $(z_{2,l}, z_{2,r})$ be two bi-free pairs of self-adjoint elements in a $*$-probability space $(\A, \varphi)$, and  the two pairs have the same hermitian bi-free central limit distribution, that is,   $\kappa_{\chi}(z_{i, \chi(1)}, ..., z_{i, \chi(n)})=\delta_{2,n}c_{\chi(1), \chi(2)}$, for $i=1, 2$, $\chi:\{1, 2, ..., n\}\rightarrow \{l,r\}$, and the second moment matrix $C=(c_{i,j})_{i, j=l,r}\ge 0$. (Definition 7.7 and Theorem 7.8 in \cite{DV}). Define
$c_l=\frac{1}{\sqrt{2}}(z_{1, l}+\imath z_{2,l})$, $c_r=\frac{1}{\sqrt{2}}(z_{1, r}+\imath z_{2,r})$, where $\imath=\sqrt{-1}$. We call $(c_l, c_r)$ a bi-circular element pair.
\end{Definition}
\begin{Theorem} A bi-circular element pair is $R$-diagonal.
\end{Theorem}
\begin{proof}
For $n\in \mathbb{N}$, $\chi:\{1,..., n\}\rightarrow \{l,r\}$, and $\omega:\{1,..., n\}\rightarrow \{1, -1\}$, let $$c_{i}=\left\{\begin{matrix}c_l, & \text{if } \chi(i)=l, \omega(i)=1,\\
c_l^*, & \text{if } \chi(i)=l, \omega(i)=-1,\\
c_r, & \text{if } \chi(i)=r, \omega(i)=1,\\
c_r^*, & \text{if } \chi(i)=r, \omega(i)=-1,
\end{matrix}\right. i=1, ..., n.$$ %
 We have
\begin{align*}
\kappa_\chi(c_1, ..., c_n)
=&\frac{1}{2^{\frac{n}{2}}}(\kappa_\chi(z_{1, \chi(1)},..., z_{1, \chi(n)})+\imath^{\omega(1)+\cdots +\omega(n)}\kappa_\chi(z_{2, \chi(1)},..., z_{2, \chi(n)}))\\
=&\frac{1}{2^{\frac{n}{2}}}\delta_{n,2}(\kappa_\chi(z_{1, \chi(1)}, z_{1, \chi(2)})+\imath^{\omega(1) +\omega(2)}\kappa_\chi(z_{2, \chi(1)}, z_{2, \chi(2)}))\\
=&\left\{\begin{matrix}c_{\chi(1),\chi(2)}, &\text{if } n=2, \text{and } \omega(1)\ne \omega(2),\\
 0, &\text{if } \omega(1)= \omega(2).
\end{matrix}\right.
\end{align*}
\end{proof}
G. Katsimpas \cite{GK} proved that if $(x_1, y_1)$ is $R$-diagonal, and $(x_1, y_1)$ and $(x_2, y_2)$ are $*$-bi-free, then $(x_1x_2, y_2y_1)$ is $R$-diagonal. If, furthermore, both $(x_1, y_1)$ and $(x_2, y_2)$ are $R$-diagonal, and the two pairs are $*$-bi-free, then $(x_1x_2, y_1y_2)$ is also $R$-diagonal (Theorem 3.2 and Proposition 3.4 in \cite{GK}). We now give formulae to compute the determining sequences of the product pairs $(x_1x_2, y_2y_1)$ and $(x_1x_2, y_1y_2)$.

\begin{Theorem} Let $(x_1, y_1)$ and $(x_2, y_2)$ be $R$-diagonal, with determining sequences $\{\alpha_{\chi}^{(1)}, \beta_{\chi}^{(1)}: \chi:\{1, ..., n\}\rightarrow \{l,r\}, n=1, 2, ...\}$ and $\{\alpha_{\chi}^{(2)}, \beta_{\chi}^{(2)}: \chi:\{1, ..., n\}\rightarrow \{l,r\}, n=1, 2, ...\}$, respectively,   and let $(x_1, y_1)$ and $(x_2, y_2)$ be $*$-bi-free in a $*$-probability space $(\A, \varphi)$.  Then the determining sequences $\{\hat{\alpha}_{\chi}: \chi:\{1,..., 2n\}\rightarrow \{l,r\}, n=1, 2,... \}$ and $\{\hat{\beta}_{ \chi}: \chi:\{1,..., 2n\}\rightarrow \{l,r\}, n=1, 2,...\}$ of the $R$-diagonal pair $(x_1x_2, y_2y_1)$ are given by the following formulae
$$\hat{\alpha}_{ \chi}=\sum_{\substack{\pi:=\pi_1\cup \pi_2\in BNC(\hat{\chi}),\\ \pi_1=\{V_1, ..., V_{p}\}, s_{\hat{\chi}}(1)\in V_1, \\ \pi_2=\{W_1,..., W_{q}\}}}
(\alpha_{\hat{\chi}}^{(1)}(V_1)\beta_{\hat{\chi}}^{(1)}(V_2)\cdots\beta_{ \hat{\chi}}^{(1)}(V_{p})\alpha_{\hat{\chi}}^{(2)}(W_1)\cdots  \alpha_{ \hat{\chi}}^{(2)}(W_{q})),$$
$$\hat{\beta}_{ \chi}=\sum_{\substack{\pi:=\pi_1\cup \pi_2\in BNC(\hat{\chi}), \\  \pi_1=\{V_1, ..., V_{p}\}, s_{\hat{\chi}}(1)\in W_1,  \\ \pi_2=\{W_1,..., W_{q}\}}}
(\beta_{\hat{\chi}}^{(2)}(W_1)\alpha_{\hat{\chi}}^{(2)}(W_2)\cdots\alpha_{ \hat{\chi}}^{(2)}(W_{q})\beta_{\hat{\chi}}^{(1)}(V_1)\cdots  \beta_{ \chi}^{(1)}(V_{p})),$$
where $\hat{\chi}$ is the canonical extension of $\chi$ to $\{1,..., 4n\}$ by the formula $\hat{\chi}(2k-1)=\hat{\chi}(2k)=\chi(k)$, for $k=1, ..., 2n$,  $$\pi_1=\{V_i\in \pi: \forall k\in V_i,  z_k\in \{x_1, x_1^*, y_1, y_1^*\},   i=1,..., p\},$$  $$\pi_2=\{W_i\in \pi: \forall k\in W_i, z_k\in \{x_2, x_2^*, y_2, y_2^*\},  i=1,..., q\}.$$
\end{Theorem}
\begin{proof} We prove the formula for $\hat{\alpha}$. The proof for $\hat{\beta}$ is essentially the same. Let $n\in \mathbb{N}$, $\chi:\{1,..., 2n\}\rightarrow \{l,r\}$, $\omega:\{1,...,n\}\rightarrow \{1,*\}$, $\omega(i)\ne \omega(i+1)$, for $i=1, 2, ..., n-1$, and $$Z_k=z_{2k-1}z_{2k}=\left\{\begin{matrix}x_1x_2, &\text{ if } \chi(k)=l,\omega(k)=1,\\
x^*_2x^*_1, & \text{ if } \chi(k)=l,\omega(k)=*,\\
y_2y_1, &\text{ if } \chi(k)=r, \omega(k)=1,\\
y_1^*y_2^*, & \text{ if }\chi(k)=r, \omega(k)=*,\end{matrix}\right. \hspace{4mm} k=1, ..., 2n.$$ 
Note that the above equation also defines $z_{2k-1}$ and $z_{2k}$, for $k=1, ..., 2n$.  By Remark 9.1.3 in \cite{CNS2}, there is an injective and partial order-preserving embedding of $BNC(\chi)$, the set of all bi-non-crossing partitions of $\{1,..., n\}$ with respect to $\chi$ (see \cite{CNS1} for the details of bi-non-crossing partitions), into $BNC(\hat{\chi})$ via $\pi\rightarrow \hat{\pi}$ where the $p$-th node of $\pi$ is replaced by $(2p-1, 2p)$. Note that  $\hat{0}_\chi=\{\{1,2\}, ..., \{4n-1, 4n\}\}$.    By Theorem 9.1.5 in \cite{CNS2}, $$\kappa_\chi(Z_1,..., Z_n)=\sum_{\pi\in BNC(\hat{\chi}), \pi\vee \hat{0}_\chi=1_{\hat{\chi}}}\kappa_\pi(z_1,..., z_{2n}).\eqno (2.1)$$  To prove the formula for $\hat{\alpha}$, we assume that $\omega(1)=1$. By $(2.1)$ and the $*$-bi-freeness of $(x_1, y_1)$ and $(x_2, y_2)$, we have
\begin{align*}
\hat{\alpha}_{\chi}=\kappa_\chi(Z_1,..., Z_{2n})=&\sum_{\pi\in BNC(\hat{\chi}), \pi\vee \hat{0}_\chi=1_{\hat{\chi}}}\kappa_\pi(z_1,..., z_{4n})\\
=&\sum_{\substack{\pi_1\cup \pi_2=\pi\in BNC(\hat{\chi}),\\ \pi\vee \hat{0}_\chi=1_{\hat{\chi}}, s_{\hat{\chi}}(1)\in V_1\in \pi_1}}\kappa_{\pi_1}(z_1',...,z_{2n}')\kappa_{\pi_2}(z_1'',..., z_{2n}''),
\end{align*}
where $z_1', ..., z_{2n}'\in \{x_1, x_1^*, y_1, y_1^*\}$, $z_1'', ..., z_{2n}''\in \{x_2, x_2^*, y_2, y_2^*\}$,
and $\pi_1$ and $\pi_2$ are the subsets of $\pi$ defined in the statement of this theorem.  Note that  every block in $\pi=\pi_1\cup \pi_2$ must contain  an even number of elements in order for $\kappa_\pi(z_1,..., z_{4n})$ to have a non-zero contribution to the sum, since both $(x_1, y_1)$ and $(x_2, y_2)$ are $R$-diagonal. By Proposition 2.11 in \cite{GK}, $\pi\vee\hat{0}_\chi=1_{\hat{\chi}}$ and $|V|$ is even for every $V\in \pi$ if and only if $s_{\widehat{\chi}}(1)\sim_\pi s_{\hat{\chi}}(4n)$, and $s_{\hat{\chi}}(2i)\sim_\pi s_{\hat{\chi}}(2i+1)$, for $i=1, ..., 2n-1$. Note also that
 $$(s_{\hat{\chi}}(2k-1), s_{\hat{\chi}}(2k))=\left\{\begin{matrix}(2s_\chi(k)-1, 2s_\chi(k)), & \text{ if } \chi(k)=l,\\
 (2s_\chi(k), 2s_\chi(k)-1), & \text{ if } \chi(k)=r, k=1,..., n.\end{matrix}\right.$$

If $0<|\chi^{-1}(\{l\})|<2n$ is even, we have $$(Z_{s_\chi(1)},..., Z_{s_\chi(2n)})=(x_1x_2,..., x_2^*x_1^*, y_2y_1, y_1^*y_2^*, ..., y_1^*y_2^*).$$ Therefore,
$$(z_{s_{\hat{\chi}}(1)},..., z_{s_{\hat{\chi}}(4n)})=(x_1, x_2, ..., x_2^*, x_1^*, y_1, y_2, ..., y_2^*, y_1^*).$$
If $|\chi^{-1}(\{l\})|$ is odd, we have $$(Z_{s_\chi(1)},..., Z_{s_\chi(2n)})=(x_1x_2,x_2^*x_1^*,..., x_1x_2, y_1^*y_2^*, y_2y_1, ..., y_1^*y_2^*).$$ Therefore,
$$(z_{s_{\hat{\chi}}(1)},..., z_{s_{\hat{\chi}}(4n)})=(x_1, x_2, ..., x_1, x_2, y_2^*, y_1^*, ..., y_2^*, y_1^*).$$
If $\chi(1)=\cdots=\chi(2n)=l$, we have $$(z_{s_{\hat{\chi}}(1)},..., z_{s_{\hat{\chi}}(4n)})=(x_1, x_2, ..., x_2^*, x_1^*).$$
It follows that  $x_1, x_1^*, y_1, y_1^*$ appear  in the positions $1, 4, 5, ..., 4k, 4k+1..., 4n$ of the sequence $$(z_{s_{\hat{\chi}}(1)},..., z_{s_{\hat{\chi}}(4n)})$$ with $z_{s_{\widehat{\chi}}(1)}=x_1$, with the $x_1$ and $x^*_1$'s appearing to the left of the $y_1$ and $y_1^*$'s, and with starred and un-starred terms appearing in alternating order. Therefore, $(z_{s_{\widehat{\chi}}(1)},..., z_{s_{\hat{\chi}}(4n)})|_{V_1}=(x_1,  x_1^*,...z_{s_{\hat{\chi}}(|V_1|)})$, where $z_{s_{\hat{\chi}}(|V_1|)}\in \{x_1^*, y_1^*\}$,  and  the element of $(z_1,...,z_{4n})$ at the position $\min_{\prec_{\hat{\chi}}}(V_i)$ is $z^*$, where $z\in \{x_1, y_1\}$, and $\min_{\prec_{\hat{\chi}}}(V_i)$ is the minimal element of $V_i$ with respect to the order $\prec_\chi$ of $\{1, ...,n\}$ defined at the beginning of Section 2,   for $i=2,..., p$.
It follows that
\begin{align*}
\kappa_{\pi_1}(z_1,..., z_{4n})=&\kappa_{\pi_1}(z_1',..., z_{2n}')\\
=&\kappa_{\hat{\chi}}((z_1',..., z_{2n}')|_{V_1})\prod_{k=2}^p\kappa_{\hat{\chi}}((z_1',..., z_{2n}')|_{V_i})\\
=&\alpha_{ \hat{\chi}}^{(1)}(V_1)\prod_{i=2}^p\beta_{ \hat{\chi}}^{(1)}(V_i).
\end{align*}
Similarly, $x_2, x_2^*, y_2, y_2^*$ appear in the positions $2, 3, ..., 4k-2, 4k-1, ..., 4n-2, 4n-1$ of the sequence $(z_{s_{\hat{\chi}}(1)},..., z_{s_{\hat{\chi}}(4n)})$, and $z_{s_{\hat{\chi}}(4k-2)}=z$, $z_{s_{\hat{\chi}}(4k-1)}=z^*$, where $z\in \{x_2, y_2\}$. It follows that the element at the position $\min_{\hat{\chi}}(W_i)$ is a non-$*$-term, since $s_{\hat{\chi}}(4k-2)$ and $s_{\hat{\chi}}(4k-1)$ must being in the same block of $\pi$ implied by the condition $\pi\vee \hat{0}_\chi=1_{\hat{\chi}}$. We thus get $$\kappa_{\pi_2}(z_1,..., z_{4n})=\kappa_{\pi_2}(z_1'',..., z_{2n}'')=\prod_{i=1}^q\alpha_{\hat{\chi}}^{(2)}(W_i).$$

If $\chi^{-1}(\{l\})=\emptyset$, we have $(z_{s_{\hat{\chi}}(1)},..., z_{s_{\hat{\chi}}}(2n))=(y_1, y_2, y_2^*, y_1^*, ..., y_2^*, y_1^*)$. We get the same formula with $z_{s_{\widehat{\chi}}(1)}=y_1$.
\end{proof}
\begin{Corollary} Under the hypotheses of Theorem 2.5, the determining sequences $\{\hat{\alpha}_{\chi, n}, \hat{\beta}_{\chi, n}: \chi:\{1, 2, ..., n\}\rightarrow \{l,r\}, n=1, 2, ...\}$ of $(x_1x_2, y_1y_2)$ are given by the following formulae.
\begin{enumerate}
 \item If  $0<|\chi^{-1}(\{l\})|<2n,$, then $\hat{\alpha}_{\chi, n}=\hat{\beta}_{\chi, n}=0$.
\item If $\chi(i)=l$ for $i=1, 2, ..., 2n$, then
$$\hat{\alpha}_{\chi}=\sum_{\substack{\pi:=\pi_1\cup \pi_2\in BNC(\hat{\chi}),\\ \pi_1=\{V_1, ..., V_{p}\}, s_{\hat{\chi}}(1)\in V_1, \\ \pi_2=\{W_1,..., W_{q}\}}}
(\alpha_{\hat{\chi}}^{(1)}(V_1)\beta_{ \hat{\chi}}^{(1)}(V_2)\cdots\beta_{ \hat{\chi}}^{(1)}(V_{p})\alpha_{ \hat{\chi}}^{(2)}(W_1)\cdots  \alpha_{ \hat{\chi}}^{(2)}(W_{q})),$$
$$\hat{\beta}_{ \chi}=\sum_{\substack{\pi:=\pi_1\cup \pi_2\in BNC(\hat{\chi}),\\ \pi_1=\{V_1, ..., V_{p}\}, s_{\hat{\chi}}(1)\in W_1, \\ \pi_2=\{W_1,..., W_{q}\}}}
(\beta_{ \hat{\chi}}^{(2)}(W_1)\alpha_{ \hat{\chi}}^{(2)}(W_2)\cdots\alpha_{ \hat{\chi}}^{(1)}(W_{q})\beta_{ \hat{\chi}}^{(1)}(V_1)\cdots  \beta_{ \chi}^{(1)}(W_{q})).
$$
\item If $\chi(i)=r$ for $i=1, 2, ..., 2n$, then
$$\hat{\alpha}_{ \chi}=\sum_{\substack{\pi:=\pi_1\cup \pi_2\in BNC(\hat{\chi}), \\ \pi_1=\{V_1, ..., V_{p}\}, s_{\hat{\chi}}(1)\in W_1, \\ \pi_2=\{W_1,..., W_{q}\}}}
(\alpha_{\hat{\chi}}^{(2)}(W_1)\beta_{ \hat{\chi}}^{(2)}(W_2)\cdots\beta_{ \hat{\chi}}^{(2)}(W_{q})\alpha_{ \hat{\chi}}^{(1)}(V_1)\cdots  \alpha_{ \widehat{\chi}}^{(1)}(V_{p})),$$
$$\hat{\beta}_{ \chi}=\sum_{\substack{\pi:=\pi_1\cup \pi_2\in BNC(\hat{\chi}),\\ \pi_1=\{V_1, ..., V_{p}\}, s_{\hat{\chi}}(1)\in V_1, \\ \pi_2=\{W_1,..., W_{q}\}}}
(\beta_{ \hat{\chi}}^{(1)}(V_1)\alpha_{ \hat{\chi}}^{(1)}(V_2)\cdots\alpha_{ \hat{\chi}}^{(1)}(V_{p})\beta_{ \hat{\chi}}^{(2)}(W_1)\cdots  \beta_{ \chi}^{(2)}(W_{q})).
$$
\end{enumerate}
 Here $\hat{\chi}$, $\pi_1$ and $\pi_2$ are those defined in Theorem 2.5.
\end{Corollary}
\begin{proof}
As in the proof of Theorem 2.5, we only prove the formulae for $\hat{\alpha}$. For $n\in \mathbb{N}$, $\chi:\{1,..., 2n\}\rightarrow \{l,r\}$, $\omega:\{1,...,n\}\rightarrow \{1,*\}$, $\omega(i)\ne \omega(i+1)$, for $i=1, 2, ..., n-1$, and $$Z_k=z_{2k-1}z_{2k}=\left\{\begin{matrix}x_1x_2, &\text{ if } \chi(k)=l,\omega(k)=1,\\
x^*_2x^*_1, & \text{ if } \chi(k)=l,\omega(k)=*,\\
y_1y_2, &\text{ if } \chi(k)=r, \omega(k)=1,\\
y_2^*y_1^*, & \text{ if }\chi(k)=r, \omega(k)=*,\end{matrix}\right. \hspace{4mm} k=1, ..., 2n.$$

When $0<|\chi^{-1}(\{l\})|<2n$, we have $$(z_{s_{\hat{\chi}}(1)},..., z_{s_{\hat{\chi}}(4n)})=(x_1, x_2, ..., x_2^*, x_1^*, y_2, y_1, ..., y_1^*, y_2^*),$$ or  $$(z_{s_{\hat{\chi}}(1)},..., z_{s_{\hat{\chi}}(4n)})=(x_1, x_2, ..., x_1, x_2, y_1^*, y_1^*, ..., y_1^*, y_2^*).$$ By the proof of Theorem 2.5, $s_{\hat{\chi}}(1)$ and $s_{\hat{\chi}}(4n)$ must be in the same block $V_1$. It follows that $\{z_k: k\in V_1\}$ contains $x_1$ and $y_2^*$. It implies that $\kappa_{\chi, V_1}(z_1, ..., z_{4n})=0$, since  $\{x_1, y_1\}$ and $\{x_2, y_2\}$ are $*$-bi-free. Thus, $\kappa_{\chi, \pi}(z_1, ..., z_{4n})=0$, for every partition $\pi$ in the sum of the formula for $\hat{\alpha}_{\chi, n}$.

When $\chi(i)=l$ for $i=1, 2, ..., 2n$, then we have $(z_{s_{\widetilde{\chi}}(1)},..., z_{s_{\widetilde{\chi}}(4n)})=(x_1, x_2, ..., x_2^*, x_1^*)$.  By the proof of Theorem 2.5,  we have
$$\hat{\alpha}_{ \chi}=\sum_{\substack{\pi:=\pi_1\cup \pi_2\in BNC(\hat{\chi}), \\ \pi_1=\{V_1, ..., V_{p}\}, s_{\hat{\chi}}(1)\in V_1, \\ \pi_2=\{W_1,..., W_{q}\}}}
\alpha_{\hat{\chi}}^{(1)}(V_1)\beta_{ \hat{\chi}}^{(1)}(V_2)\cdots\beta_{ \hat{\chi}}^{(1)}(V_{p})\alpha_{ \hat{\chi}}^{(2)}(W_1)\cdots  \alpha_{ \hat{\chi}}^{(2)}(W_{q}).
$$
When $\chi(i)=r$ for $i=1, 2, ..., 2n$, then we have $(z_{s_{\widetilde{\chi}}(1)},..., z_{s_{\widetilde{\chi}}(4n)})=(y_2, y_1, ..., y_1^*, y_2^*)$. By the proof of Theorem 2.5, we have
 $$\hat{\alpha}_{ \chi}=\sum_{\substack{\pi:=\pi_1\cup \pi_2\in BNC(\hat{\chi}), \\ \pi_1=\{V_1, ..., V_{p}\}, s_{\hat{\chi}}(1)\in W_1, \\ \pi_2=\{W_1,..., W_{q}\}}}
\alpha_{\hat{\chi}}^{(2)}(W_1)\beta_{ \hat{\chi}}^{(2)}(W_2)\cdots\beta_{ \hat{\chi}}^{(2)}(W_{q})\alpha_{ \hat{\chi}}^{(1)}(V_1)\cdots  \alpha_{ \hat{\chi}}^{(1)}(V_{p}).
$$
\end{proof}

It was proved that if $a$ is a $R$-diagonal, then $a^*a$ and $aa^*$ are free (Corollary 15.11 in \cite{NS}). In bi-free probability, G. Katsimpas showed that if $(x,y)$ is $R$-diagonal, then $(xx^*, y^*y)$ and $(x^*x, yy^*)$ are bi-free (Proposition 3.6 in \cite{GK}). The following result shows that it is not necessarily true that  $(xx^*, yy^*)$ and $(x^*x,y^*y)$ are bi-free. Another counterexample was given in \cite{GK} (Example 3.7 in \cite{GK}).
\begin{Theorem} There is an $R$-diagonal pair $(x,y)$ of random variables in a $*$-probability space $(\A, \varphi)$ such that $(xx^*, yy^*)$ and $(x^*x,y^*y)$ are not bi-free.
\end{Theorem}
\begin{proof} By Section 7 of \cite{GS}, there is a pair $(x,y)$ of random variables in a $*$-probability space such that $\varphi(x)=\varphi(y)=0$ and $\kappa_2(x, y^*)=\kappa_2(x^*, y)=1$, and $\kappa_\chi(x,x^*, y, y^*)=0$ for all $\chi:\{1, ..., n\}\rightarrow \{l,r\}$ and $n\ge 3$. It implies that $$\kappa_2(y, x^*)=\varphi(yx^*)=\overline{\varphi(x,y^*)}=\overline{\kappa_2(x, y^*)}=1.$$ Similarly, $\kappa_2(y^*, x)=1$. Therefore,   its $R$-transform   $\mathcal{R}_{x, x^*, y, y*}=z_lz_r^*+z_rz_l^*+z_l^*z_r+z_r^*z_l$.  Let $\chi:(1,2)\mapsto (l,r)$, and $\hat{\chi}:(1, 2, 3, 4)\mapsto (l, l, r, r)$, $\hat{0}_\chi=\{\{1, 2\}, \{3, 4\}\}$ . By Theorem 9.1.5 in \cite{CNS2}, we have $$\kappa_\chi(xx^*, y^*y)=\sum_{\pi\in BNC(\hat{\chi}), \pi\vee \hat{0}_\chi=1_{\hat{\chi}}}\kappa_{\pi}(x, x^*, y^*,y)=k_{\hat{\chi}}(x, x^*, y^*, y)+\kappa_2(x^*, y)\kappa_2(x, y^*)=1,$$ where we used the fact that $\kappa_{\hat{\chi}}(x^*,y)=\kappa_2(x^*, y)$, $\kappa_{\hat{\chi}}(x,y^*)=\kappa_2(x, y^*)$, and the only partitions $\pi\in BNC(\hat{\chi})$ with possible non-zero contribution to the sum are those consisting of only even size blocks, since $(x,y)$ is $R$-diagonal. It implies that $(xx^*, yy^*)$ and $(x^*x, y^*y)$ are not bi-free.
\end{proof}
\section{$R$-diagonal pairs of left and right acting operators}

In this section, we focus on the study of $R$-diagonal pairs of left acting and right acting operators from finite von Neumann algebras in the standard form, giving characterizations of the $R$-diagonal pairs in terms of the $*$-distributions of the random variables, and the distributional invariance under multiplication by free unitaries.

Let $(\A,\varphi)$ be a $W^*$-probability space, that is, a von Neumann algebra $\A$ with a faithful normal tracial state $\varphi$ on $\A$. Represent $\A$ into $B(L^2(\A, \varphi))$ in two ways: $L:\A\rightarrow B(L^2(\A, \varphi))$ and $R:\A^{op}\rightarrow B(L^2(\A, \varphi))$, the left and, respectively, right multiplications of $\A$ on $\A\subset L^2(\A, \varphi)$. Define $\varphi(T)=\langle Te, e\rangle_\varphi, \forall T\in B(L^2(\A, \varphi))$, where $e\in \A$ is the identity operator in $\A$. Then $(B(L^2(\A,\varphi)), \varphi)$ is a $C^*$-probability space, and $L$ and $R$ are faithful $*$-representations of $\A$ and $\A^{op}$, respectively.

In this section, we always assume that $(\A, \varphi)$ is a $W^*$-probability space.

 Let $I$ and $J$ be two disjoint index sets, and $((z_i)_{i\in J}, (z_j)_{j\in J})$ a two-faced family of random variables in $\A$. Let $Z_i=L(z_i)$ for $i\in I$, and $Z_j=R(z_j)$ for $j\in J$.    Let $\chi:\{1, 2, \cdots, n\}\rightarrow I\bigsqcup J$.  The permutation $s_\chi $ (defined at the beginning of Section 2) induces a lattice isomorphism from $NC(n)$ onto $BNC(\chi)$ by $\pi \mapsto s_\chi\cdot\pi$, for $\pi\in NC(n)$, where $$s_\chi\cdot \pi=\{s_\chi\cdot V=\{s_\chi(t_1), s_\chi(t_2), \cdots, s_\chi(t_k)\}: V=\{t_1, t_2, \cdots, t_k\}\in \pi\}.$$ Thus, $\mu_{BNC}(s_\chi\circ \pi, 1_n)=\mu_{NC}(\pi, 1_n)$, for $\pi\in NC(n)$. For a subset $V=\{i_1< ...< i_p\}\subseteq \{1, ..., n\}$,  we define $$\varphi_V(z_{\chi(1)}, ..., z_{\chi(n)})=\varphi(z_{\chi(i_1)}\cdots z_{\chi(i_p)}).$$

  By the definitions of representations $L$ and $R$, we have $$\varphi(Z_{\chi(1)}\cdots Z_{\chi(n)})=\varphi(z_{\chi(s_\chi(1))}\cdots z_{\chi(s_\chi(n))}). \eqno (3.1)$$ For $ V=\{i_1<\cdots <i_k\}\subseteq \{1, 2, ..., n\}$, let $$s_\chi\circ V=\{s_\chi(i_1)\prec_\chi\cdots\prec_\chi s_\chi(i_k)\}=\{j_1<\cdots <j_k\}.$$ We then have
\begin{align*}
 \varphi_{s_\chi\circ V}(Z_{\chi(1)},..., Z_{\chi(n)})=&\varphi(Z_{\chi(j_1)}\cdots Z_{\chi(j_k)})\\
 =&\varphi(z_{\chi(s_\chi(i_1))}\cdots z_{\chi(s_\chi(i_k))})\\
 =&\varphi_V(z_{(\chi\circ s_\chi)(1))},...,z_{(\chi\circ s_\chi)(n)}),
\end{align*}
where $\chi\circ s_\chi: \{1, ..., n\}\rightarrow I\bigsqcup J$ is the composition of $s_\chi$ and $\chi$, that is, $\chi\circ s_\chi(i)=\chi(s_\chi(i))$, for $i=1, ..., n$.
  It implies that
\begin{align*}
\kappa_n(z_{(\chi\circ s_\chi)(1)}, ..., z_{(\chi\circ s_\chi)(n)})=&\sum_{\pi\in NC(n)}\varphi_\pi(z_{\chi(s_\chi(1))},..., z_{\chi(s_\chi(n))})\mu_{NC}(\pi, 1_n)\\
=&\sum_{\pi\in NC(n)}\prod_{V\in \pi}\varphi_{V}(z_{\chi(s_\chi(1))}, ..., z_{\chi(s_\chi(n))})\mu_{NC}(\pi,1_n)\\
=&\sum_{\pi\in NC(n)}\prod_{V\in\pi}\varphi_{s_\chi\circ V}(Z_{\chi(1)},..., Z_{\chi(n)})\mu_{NC}(\pi,1_n)\\
=&\sum_{\sigma\in BNC(\chi)}\varphi_\sigma(Z_{\chi(1)}, \cdots, Z_{\chi(n)})\mu_{BNC}(\sigma,1_n)\\
=&\kappa_\chi(Z_{\chi(1)}, \cdots, Z_{\chi(n)}).
\end{align*}
We thus have  $$\kappa_\chi(Z_{\chi(1)}, \cdots, Z_{\chi(n)})=\kappa_n(z_{(\chi\circ s_\chi)(1)}, \cdots, z_{(\chi\circ s_\chi)(n)})=\kappa_n(z_{\chi(s_\chi(1))}, \cdots, z_{\chi(s_\chi(n))}). \eqno (3.2)$$

The $*$-distribution of an $R$-diagonal random variable $a$ in a $*$-probability space $(\A, \varphi)$ is uniquely determined by the distributions of $a^*a$ and $aa^*$ (Corollary 15.7 in \cite{NS}). In the $R$-diagonal pair case, we have the following similar result.
\begin{Proposition} Let $x$ and  $y$ be random variables in $(\A,\varphi)$. If $(L(x),R(y))$ is $R$-diagonal, then the $*$-distribution of $(L(x),R(y))$ is determined by the distributions of $(xx^*, ..., yy^*), (x^*x, ..., y^*y)$,  where   $k_1$  arguments in the tuple $(x, x^*, ..., y, y^*)$ (or $(x^*, x, ..., y^*, y)$)  are from  $\{x, x^*\}$, and $k_2$  arguments from  $\{y,y^*\}$, $k_1, k_2\ge 0$, $k_1+k_2=2n, n=1, 2,...$. Precisely, for operators $x_1, x_2, y_1, y_2$ in $(\A,\varphi)$, if $$(L(x_1),R(y_1)), (L(x_2), R(y_2))\in B(L^2(\A,\varphi))$$ are $R$-diagonal, and $$\kappa_n(x_1x_1^*,..., y_1y_1^*)=\kappa_{n}(x_2x_2^*,..., y_2y_2^*), \kappa_n(x_1^*x_1,..., y_1^*y_1)=\kappa_n(x_2^*x_2,..., y_2^*y_2), \eqno (3.3)$$n=1, 2, ..., then $(L(x_1), R(y_1))$ and $(L(x_2), R(y_2))$ are identically $^*$-distributed.
\end{Proposition}
\begin{proof} Let $\alpha_{1; k_1, k_2}=\kappa_{2n}(x, x^*, ..., y, y^*), \alpha_{2; k_1, k_2}=\kappa_{2n}(x^*, x, ..., y^*, y),$ with $k_1$ arguments from $\{x, x^*\}$ and $k_2$ arguments from $\{y, y^*\}$ in the tuple $$(x, x^*, ..., y, y^*), \text{or } (x^*, x, ..., y^*, y)$$ such that $k_1+k_2=2n$ and $k_1, k_2\ge 0$.
By $(3.2)$, the $*$-distribution of $(L(x),R(y))$ is determined by $\alpha_{1; k_1, k_2}$, and $\alpha_{2; k_1, k_2},$ for $k_1+k_2=2n, n=1, 2, ...$.

For a subset $V=\{i_1<i_2<...<i_k\}\subset \{1, ..., n\}$, let $k_1(V)$ and $k_2(V)$ be numbers of arguments from $\{x, x^*\}$ and, respectively,  arguments from $\{y, y^*\}$ in the tuple $(x, x^*,..., y, y^*)|_{\widetilde{V}}$ or $(x^*, x, ..., y^*, y)|_{\widetilde{V}}$, where $\widetilde{V}=\{2i_1-2, 2i_1-1, 2i_2-2, 2i_2-1,..., 2i_k-2, 2i_k-1\}$, if $i_1\ne 1$; $\widetilde{V}=\{1,  2i_2-2, 2i_2-1,..., 2i_k-2, 2i_k-1, 2n\}$, if $i_1=1$.
Moreover, the mapping $V\mapsto \widetilde{V}$, $\pi\mapsto \widetilde{\pi}=\{\widetilde{V}:V\in \pi\}$, induces a bijection from $NC(n)$ onto the following set
 $$\P=\{\pi\in NC(2n): 1\sim_{\pi}2n, 2\sim_{\pi}3, ..., 2n-2\sim_{\pi}2n-1\}$$ (see the discussion on the top of Page 189 in \cite{NS}).

 Note that for a block $V\in \pi$, $\pi\in NC(2n)$, $\kappa_{V}(x, x^*, ..., y, y^*)$ has one of the following forms $$\kappa_{|V|}(x^{\omega(1)},... x^{\omega(k)}, y^{\omega(k+1)}..., x^{\omega(|V|)}), 0\le k\le n, \omega:\{1,..., n\}\rightarrow \{1, *\}.$$ It implies from (3.2) that $\kappa_V(x, x^*, ..., y^*)=\kappa_\chi(Z_1, ..., Z_{|V|})$, for some $\chi:\{1, ..., |V|\}\rightarrow \{l,r\}$, $$Z_i\in\left\{\begin{matrix} \{L(x), L(x^*)\}, & \text{if } \chi(i)=l,\\
    \{R(y), R(y^*)\}, & \text{if }\chi(i)=r,\end{matrix}\right.  \hspace{4mm} i=1, ..., |V|.$$

 It follows that $\kappa_V(x, ..., y^*)=0$ if $|V|$ is odd, since $(L(x), R(y))$ is $R$-diagonal. Let $$\sigma=\{\{1,2\}, \{3, 4\}, ..., \{2n-1, n\}\}.$$  By Proposition 2.11 in \cite{GK}, $$\{\pi\in NC(2n): \pi\vee \sigma=1_{2n}, |V| \text{ is even},  \forall V\in \pi \}=\P.$$
  We thus have
\begin{align*}
\kappa_{n}(xx^*,...,yy^*)=&\sum_{\pi\in NC(2n), \pi\vee\sigma=1_{2n}}\kappa_\pi(x, x^*,..., y, y^*)\\
=&\sum_{\pi\in \P}\kappa_\pi(x, x^*,..., y, y^*)\\
=&\sum_{\pi\in NC(n), 1\in V_1, \pi=\{V_1, ..., V_d\}}\alpha_{1;, k_1(V_1), k_2(V_1)}\alpha_{2, k_1(V_2), k_2(V_2)}\cdots \alpha_{2; k_1(V_d), k_2(V_d)}\\
=&\alpha_{1, k_1, k_2}\\
+&\sum_{\pi\in NC(n),\pi\ne 1_{n}, 1\in V_1, \pi=\{V_1, ..., V_d\}}\alpha_{1;, k_1(V_1), k_2(V_1)}\alpha_{2, k_1(V_2), k_2(V_2)}\cdots \alpha_{2; k_1(V_d), k_2(V_d)},
\end{align*}
It implies that $$\kappa_{2n}(x, x^*,..., y, y^*)=\kappa_n(xx^*,..., yy^*)$$
$$-\sum_{\pi\in NC(n)\setminus \{1_n\},1\in V_1, \pi=\{V_1, ..., V_d\}}\alpha_{1;, k_1(V_1), k_2(V_1)}\alpha_{2, k_1(V_2), k_2(V_2)}\cdots \alpha_{2; k_1(V_d), k_2(V_d)}. \eqno (3.4)$$ Very Similarly, we have $$\kappa_{2n}(x^*, x, ..., y^*, y)=\kappa_n(x^*x, ..., y^*y)$$
$$-\sum_{\pi\in NC(n)\setminus \{1_n\}, 1\in V_1, \pi=\{V_1, ..., V_d\}}\alpha_{2;, k_1(V_1), k_2(V_1)}\alpha_{1, k_1(V_2), k_2(V_2)}\cdots \alpha_{1; k_1(V_d), k_2(V_d)}.\eqno (3.5)$$
When $n=1$, by (3.4) and (3.5), $\kappa_2(x,x^*)=\kappa(xx^*), \kappa_2(x,y^*)=\kappa(xy^*), \kappa_2(y,y^*)=\kappa(yy^*)$, and $\kappa_2(x^*,x)=\kappa(x^*x), \kappa_2(x^*,y)=\kappa(x^*y), \kappa_2(y^*,y)=\kappa(y^*y)$. Suppose that there are polynomials $P_{m, k_1, k_2}$ and $Q_{m, k_1, k_2}$, independent of the choice of $x$ and $y$, such that  $$\alpha_{1; k_1, k_2}=P_{m, k_1, k_2}(\kappa_i(xx^*,..., yy^*),\kappa_i(x^*x, ..., y^*y), i=1, 2, ..., m),$$
$$ \alpha_{2; k_1, k_2}=Q_{m, k_1, k_2}(\kappa_i(xx^*,..., yy^*),\kappa_i(x^*x, ..., y^*y), i=1, 2, ..., m),$$  for $k_1+k_2=2m<2n$. By (3.4) and (3.5), there are polynomials $P_{n, k_1, k_2}$ and $Q_{n, k_1, k_2}$, independent of the choice of $x$ and $y$,  such that $$\alpha_{1; k_1, k_2}=P_{n, k_1, k_2}(\kappa_i(xx^*,..., yy^*),\kappa_i(x^*x, ..., y^*y), i=1, 2, ..., n),$$ and
$$ \alpha_{2; k_1, k_2}=Q_{n, k_1, k_2}(\kappa_i(xx^*,..., yy^*),\kappa_i(x^*x, ..., y^*y), i=1, 2, ...,n),$$  for $k_1+k_2=2n$. It follows from $(3.2)$ and $(3.3)$ that $(L(x_1), R(y_1))$ and $(L(x_2), R(y_2))$ are identically $*$-distributed.

\end{proof}

We will use the following sets
\begin{align*}
P_{1,1}(x,y)=&\{x^{\omega(1)}x^{\omega(2)}\cdots x^{\omega(k_1)}y^{\omega(k_1+1)}\cdots y^{\omega(k_2)}:k_1+k_2=2n-1, k_1, k_2\ge 0,\\
 &\omega:\{1, ..., 2n-1\}\rightarrow \{1,*\}, \omega(1)=*, \omega(i)\ne \omega(i+1), i=1, 2, ..., 2n-2, n=1, 2,...\},\\
P_{1,2}(x,y)=&\{z-\varphi(z):z=x^{\omega(1)}x^{\omega(2)}\cdots x^{\omega(k_1)}y^{\omega(k_1+1)}\cdots y^{\omega(k_2)}, k_1+k_2=2n, k_1, k_2\ge 0,\\
 &\omega:\{1, ..., 2n\}\rightarrow \{1,*\}, \omega(1)=*, \omega(i)\ne \omega(i+1), i=1, 2, ..., 2n-1, n=1, 2,...\},
 \end{align*}
 and
 \begin{align*}
 P_{2,2}(x,y)=&\{x^{\omega(1)}x^{\omega(2)}\cdots x^{\omega(k_1)}y^{\omega(k_1+1)}\cdots y^{\omega(k_2)}:k_1+k_2=2n-1, k_1, k_2\ge 0,\\
 &\omega:\{1, ..., 2n-1\}\rightarrow \{1,*\}, \omega(1)=1, \omega(i)\ne \omega(i+1), i=1, 2, ..., 2n-2, n=1, 2,...\},\\
P_{2,1}(x,y)=&\{z-\varphi(z): z=x^{\omega(1)}x^{\omega(2)}\cdots x^{\omega(k_1)}y^{\omega(k_1+1)}\cdots y^{\omega(k_2)}, k_1+k_2=2n, k_1, k_2\ge 0,\\
 &\omega:\{1, ..., 2n\}\rightarrow \{1,*\}, \omega(1)=1, \omega(i)\ne \omega(i+1), i=1, 2, ..., 2n-1, n=1, 2,...\}.
 \end{align*}
 \begin{Theorem} A pair $(L(x),R(y))$ of elements in  $B(L^2(\A, \varphi))$ is $R$-diagonal if and only if  $$\varphi(p_{i_1,i_2}p_{i_2, i_3}\cdots p_{i_{n-1}, i_n})=0,\eqno (3.6)$$ for $p_{i_k, i_{k+1}}\in P_{i_k, i_{k+1}}$, $i_k, i_{k+1}\in \{1,2\}$,   $j=1,2,...,n-1$, with the following conditions
   \begin{enumerate}
   \item there exists a number $m$ such that $p_{i_1, i_2},..., p_{i_m,i_{m+1}}$ only contain  factors of $x, x^*$ and $$p_{i_{m+1},i_{m+2}},..., p_{i_{n-1},i_n}$$ only contain factors of $y, y^*$, where $0\le m\le n-1$, or
   \item there exists a number $m$ such that $p_{i_1, i_2},..., p_{i_{m-1},i_{m}}$ only contain  factors of $x, x^*$, $p_{i_m, i_{m+1}}$ contains both factors from $\{x,x^*\}$ and $\{y, y^*\}$, and $p_{i_{m+1},i_{m+2}},..., p_{i_{n-1},i_n}$ only contain factors of $y, y^*$,
   \end{enumerate} $k=1, 2, ..., n-1$, and $n=2, 3,  ...$.
 \end{Theorem}
 \begin{proof}
  Suppose that $(L(x),R(y))$ is $R$-diagonal. Let $u$ be a Haar unitary, and $\{u\}$ and $\{x,y\}$ be $*$-free in $(\A, \varphi)$.

  We first prove $(3.6)$ for $(ux, uy)$. In this case, $$p_{1,1}(u)=p_{1,1}u^*, p_{1,2}(u)=p_{1,2}, p_{2,1}(u)=up_{2,1}u^*, p_{2,2}(u)=up_{2,2},$$ where $p_{i,j}(u)\in P_{i,j}(ux,uy)$, $p_{i,j}\in P_{i,j}(x,y)$. Note that $\varphi(p_{1,1})=\varphi(p_{2,2})=0$, since $(L(x),R(y))$ is $R$-diagonal. Moreover, the product $p_{i,j}(u)p_{j,l}(u)$ is obtained by formally putting the two quantities in the  ordered pair $(p_{i,j}(u), p_{j,l}(u))$ together without any cancellations. Precisely, we have
   \begin{align*}p_{1,1}(u)p_{1,1}(u)=p_{1,1}u^*p_{1,1}u^*,&\ \  p_{1,1}(u)p_{1,2}(u)=p_{1,1}u^*p_{1,2},\\
   p_{2,1}(u)p_{1,1}(u)=up_{2,1}u^*p_{1,1}u^*,& \ \ p_{2,1}(u)p_{1,2}(u)=up_{2,1}u^*p_{1,2},\\
    p_{1,2}(u)p_{2,1}(u)=p_{1,2}up_{2,1}u^*, &\ \ p_{2,2}(u)p_{2,1}(u)=up_{2,2}up_{2.1}u^*, \\
  p_{2,2}(u)p_{2,2}(u)=up_{2,2}up_{22}, & \ \ p_{1,2}(u)p_{2,2}(u)=p_{1,2}up_{2,2}.
  \end{align*}
  It implies that  elements from $\bigcup_{i,j=1}^2 P_{i,j}(x,y)$ and  elements from $\{u, u^*\}$ appear alternatingly in the product $$p_{i_1, i_2}(u) p_{i_2, i_3}(u) \cdots p_{i_{n-1}, i_n}(u).$$ Then (3.6)  follows for $p_{i,j}(u)$'s, since $\varphi(p_{i,j})=\varphi(u)=\varphi(u^*)=0$, for $i, j=1, 2$, and $\{u\}$ and $\{x,y\}$ are $*$-free.

 By $(3.1)$, we get  $$\varphi(x^{\omega(1)}\cdots x^{\omega(k)}y^{\omega(k+1)}\cdots y^{\omega(n)})=\varphi(Z_{\chi(1)}\cdots Z_{\chi(n)}),\eqno(3.7)$$ where $\chi:\{1, ..., n\}\rightarrow \{l^*, l, r, r^*\}$, $Z_{l^*}=L(x^*), Z_{l}=L(x), Z_{r^*}=R(y^*)$ and $Z_r=R(y)$, $\omega:\{1, ..., n\}\rightarrow \{*, 1\}$, $\chi$ is chosen so that $$\varphi(x^{\omega(1)}\cdots x^{\omega(k)}y^{\omega(k+1)}\cdots y^{\omega(n)})=\varphi(z_{\chi(s_\chi(1))}\cdots z_{\chi(s_\chi(n))}),$$ where $Z_{\chi(i)}=L(z)$, if $\chi(i)\in\{l, l^*\}$; $Z_{\chi(i)}=R(z)$, if $\chi(i)\in\{r, r^*\}$, for $i=1,..., n$ (see the discussion at the beginning of this section). Let $$Z_{\chi(i)}(u)=\left\{\begin{matrix}L(ux), &\text{ if } \chi(i)=l,\\
L(x^*u^*), & \text{ if } \chi(i)=l^*,\\
R(uy), &     \text{ if } \chi(i)=r,\\
R(y^*u^*), & \text{ if } \chi(i)=r^*,
\end{matrix}
\right. \hspace{4mm} i=1, ..., n. $$ Similarly, by (3.1), we have $$\varphi((ux)^{\omega(1)}\cdots (ux)^{\omega(k)}(uy)^{\omega(k+1)}\cdots (uy)^{\omega(n)})=\varphi(Z_{\chi(1)}(u)\cdots Z_{\chi(n)}(u)).\eqno(3.8)$$

By Corollary 2.18 and Theorem 4.4 in \cite{GK}, $(L(x), R(y))$ and $(L(ux), R(uy))$ are identically $*$-distributed. It implies from this fact and $(3.7)$ and $(3.8)$ that $$\varphi(p_{i_1,i_2}p_{i_2, i_3}\cdots p_{i_{n-1}, i_n})=\varphi(p_{i_1,i_2}(u)p_{i_2, i_3}(u)\cdots p_{i_{n-1}, i_n}(u)),$$ where $p_{ij}$ and $p_{ij}(u)$ are chosen to satisfy the conditions in the statement of this theorem.  We have proved that $(ux, uy)$ satisfies $(3.6)$. Hence, $(x,y)$ satisfies $(3.6)$.

  Conversely, suppose that $(x,y)$ satisfies $(3.6)$.    Let $u$ be a Haar unitary, and $\{u\}$ and $\{x,y\}$ be $*$-free in $(\A, \varphi)$.  By Corollary 2.18 and Theorem 3.2 in \cite{GK}, and the first part of the current proof, $(ux, uy)$ satisfies $(3.6)$.   It is obvious that
   $$\varphi(\underbrace{(ux)(ux)^*\cdots (uy)(uy)^*}_{2n})=\varphi(uxx^*xx^*\cdots yy^*u^*)=\varphi(\underbrace{xx^*\cdots yy^*}_{2n}),$$
  $$\varphi(\underbrace{(ux)^*(ux)\cdots (uy)^*(uy)}_{2n})=\varphi(\underbrace{x^*x\cdots y^*y}_{2n}),  n=1,2,...\eqno (3.9)$$
  For $\chi:\{1,2,..., n\}\rightarrow \{l,r\}$ and $\omega:\{1,2,..., n\}\rightarrow \{1,*\}$, $n\in\mathbb{N}$, let
  $$z_k=\left\{\begin{matrix}
  x,&\omega(k)=1, \chi(k)=l\\
  x^*, & \omega(k)=*, \chi(k)=l\\
  y,& \omega(k)=1, \chi(k)=r\\
  y^*,&\omega(k)=1, \chi(k)=l\\
  \end{matrix},\right. \hspace{3mm}
  w_k=\left\{\begin{matrix}
  uz_k,&\omega(k)=1\\
  z_k^*u^*, & \omega(k)=*
  \end{matrix},\right.
  $$
  $$Z_k=L(z_k), W_k=L(w_k), \text{ if } \chi(k)=l; Z_k=R(z_k), W_k=R(w_k), \text{ if } \chi(k)=r,\eqno (3.10)$$
   for $k=1,..., n$.
   We  prove that $$\varphi(Z_1\cdots Z_n)=\varphi(W_1\cdots W_n).$$
  By $(3.1)$, it sufficient to prove $$\varphi(z_{s_\chi(1)}\cdots z_{s_\chi(n)})=\varphi(w_{s_\chi(1)}\cdots w_{s_\chi(n)}),$$ that is, $$\varphi(x^{\omega_1(1)}\cdots x^{\omega_1(k)} y^{\omega_1(k+1)}\cdots y^{\omega(n)})=\varphi((ux)^{\omega_1(1)}\cdots (ux)^{\omega_1(k)} (uy)^{\omega_1(k+1)}\cdots (uy)^{\omega_1(n)}), \eqno (3.11)$$ where $0\le k\le n, \omega_1(i)=\omega(s_\chi(i))$, for $i=1, 2, ..., n$.

  When $n=1$, $\varphi(x^{\omega(1)})=0=\varphi(y^{\omega(1)})=\varphi((ux)^{\omega(1)})=\varphi((uy)^{\omega(1)})$, since $\{x,y\}$ and $\{ux, uy\}$ satisfy $(3.6)$. Suppose that $(3.11)$ is true when  $n<m$. Now we prove $(3.11)$ when $n=m$.

  We adopt some ideas from the proof of Proposition 2.1 in \cite{NSS}. We take the product  $$x^{\omega_1(1)}\cdots x^{\omega_1(k)} y^{\omega_1(k+1)}\cdots y^{\omega(n)}$$ and form an interval partition of the ordered set of $n$ factors of the product   by the following rule.   The $i$-th factor and the $i+1$-th factor are in two adjacent blocks  for $1\le i<n$ if $\omega_1(i)=\omega_1(i+1)$. Then we have $$x^{\omega_1(1)}\cdots x^{\omega_1(k)} y^{\omega_1(k+1)}\cdots y^{\omega(n)}=\prod_{r=1}^t(p_{j_r, j_{r+1}}+\lambda_r),$$
 for some $t\ge 1$, $j_1, j_2, ..., j_{t+1}\in \{1, 2\}$, where $p_{i,j}\in P_{i,j}(x,y)$, $\lambda_r$ is determined as follows:  if $j_r=j_{r+1}$, then $\lambda_{r}=0$; if $j_r\ne j_{r+1}$, $\lambda_r=\varphi(z)$, where $p_{j_r, j_{r+1}}=z-\varphi(z)$. Similarly, we have
 $$(ux)^{\omega_1(1)}\cdots (ux)^{\omega_1(k)} (uy)^{\omega_1(k+1)}\cdots (uy)^{\omega(n)}=\prod_{r=1}^t(p_{j_r, j_{r+1}}(u)+\lambda_r(u)),$$ where $\lambda_r(u)=\lambda_r$, for $r=1,..., t+1$, by $(3.9)$.
 It implies from $(3.6)$ that
\begin{align*}
&\varphi(x^{\omega_1(1)}\cdots x^{\omega_1(k)} y^{\omega_1(k+1)}\cdots y^{\omega(n)})\\
=&\varphi(\prod_{r=1}^{r=t}p_{j_r, j_{r+1}})\\
+&\sum_{\emptyset \ne S\subseteq \{1,..., t\}}(\prod_{r\in S}\lambda_r)\varphi(\prod_{r\in \{1,..., t\}\setminus S}p_{j_r, j_{r+1}})\\
=&\sum_{\emptyset \ne S\subseteq \{1,..., t\}}(\prod_{r\in S}\lambda_r)\varphi(\prod_{r\in \{1,..., t\}\setminus S}p_{j_r, j_{r+1}}),
\end{align*}
\begin{align*}
&\varphi((ux)^{\omega_1(1)}\cdots (ux)^{\omega_1(k)} (uy)^{\omega_1(k+1)}\cdots (uy)^{\omega(n)})\\
=&\varphi(\prod_{r=1}^{r=t}p_{j_r, j_{r+1}}(u))\\
+&\sum_{\emptyset \ne S\subseteq \{1,..., t\}}(\prod_{r\in S}\lambda_r)\varphi(\prod_{r\in \{1,..., t\}\setminus S}p_{j_r, j_{r+1}}(u))\\
=&\sum_{\emptyset \ne S\subseteq \{1,..., t\}}(\prod_{r\in S}(\lambda_r))\varphi(\prod_{r\in \{1,..., t\}\setminus S}p_{j_r, j_{r+1}}(u)),
\end{align*}
where the multiplication orders in $\prod_{r\in \{1,..., t\}\setminus S}p_{j_r, j_{r+1}}$ and $\prod_{r\in \{1,..., t\}\setminus S}p_{j_r, j_{r+1}}(u)$ are derived  from $p_{j_1,j_2}p_{j_2, j_2}\cdots p_{j_t, j_{t+1}}$ and, respectively, from  $p_{j_1,j_2}(u)p_{j_2, j_2}(u)\cdots p_{j_t, j_{t+1}}(u)$ by removing the factors with indices in $S$. Therefore,  $\prod_{r\in \{1,..., t\}\setminus S}p_{j_r, j_{r+1}}$ is a linear combination of terms $$v_p:=x^{\omega_1(1)}\cdots x^{\omega_1(k)} y^{\omega_1(k+1)}\cdots y^{\omega(p)}$$ with length $p$ less that $n$, i. e,, $\prod_{r\in \{1,..., t\}\setminus S}p_{j_r, j_{r+1}}=\sum_{j=1}^d\beta_jv_{p_j}$. Similarly, $$\prod_{r\in \{1,..., t\}\setminus S}p_{j_r, j_{r+1}}=\sum_{j=1}^d\beta_jv_{p_j}(u).$$ By the inductive hypothesis, we have $$\varphi(\prod_{r\in \{1,..., t\}\setminus S}p_{j_r, j_{r+1}})=\varphi(\prod_{r\in \{1,..., t\}\setminus S}p_{j_r, j_{r+1}}(u)).$$ It implies that $(3.11)$ is true when $n=m$. We have proved that $\varphi(Z_1\cdots Z_n)=\varphi(W_1\cdots W_n)$, for $n=1,2,...$, which means that  $\{L(x),R(y)\}$ and $(L(ux),R(uy))$ have the same $*$-distribution. By  Theorem 4.4 in \cite{GK},  $(L(x), R(y))$, is $R$-diagonal.
 \end{proof}
Summarizing Theorem 3.2 and the work in \cite{GK}, we get the following result, which characterizes $R$-diagonal pairs of left and right operators in terms of  the $*$-moments of the random variables, and of the distributional invariance of the random variables under multiplication by free (Haar) unitaries. The following theorem is a bi-free analogue of the main result (Theorem and Definition 1.2) of \cite{NSS}.
\begin{Theorem} For $x,y\in \A$, the following statements are equivalent.
\begin{enumerate}
\item The pair $(L(x), R(y))$ is $R$-diagonal in $(B(L^2(\A,\varphi)),\varphi)$.
\item The elements $x$ and $y$ satisfy $(3.6)$.
\item Let $u$ be a Haar unitary, and $u$ and $\{x,y\}$ are $*$-free in $\A$. Then $(L(x),R(y))$ and $(L(ux), R(uy))$ have the same $*$-distribution.
\item Let $u$ be a  unitary in $\A$, and $u$ and $\{x,y\}$ are $*$-free in $\A$. Then $(L(x),R(y))$ and $(L(ux), R(uy))$ have the same $*$-distribution.
\end{enumerate}
\end{Theorem}
\begin{proof} The equivalence of  $(1)$ and $(2)$ is proved in Theorem 3.2. The equivalence of $(1)$ and $(3)$ was proved in Corollary 2.18 and Theorem 4.4 in \cite{GK}. If $(4)$ holds true, then so does $(3)$. Therefore, $(L(x), R(y))$ is $R$-diagonal. Conversely, suppose $(L(x), R(y))$ is $R$-diagonal. Let $\{v\}$ be a Haar unitary such that $\{v\}$, $\{u\}$ and $\{x,y\}$ are $*$-free in $(\A,\varphi)$. Then $u$ and $\{vx, vy\}$ are $*$-free. We can realize $\A$ as $(\A,\varphi)=(\A_1,\varphi_1)*(\A_2,\varphi_2)$, the reduced free product of the two $W^*$-probability spaces $(\A_1, \varphi_1)$ and $(\A_2, \varphi_2)$ such that $u\in \A_1$ and $v, x,y\in \A_2$. By 6.2 in \cite{DV}, $(L(u), R(u))$ and $(L(vx),R(vy))$ are $*$-bi-free. Let $$U_l=L(u), U_r=R(u), U=(U_l, U_r), V_l=L(vx), V_r=R(vy), V=(V_l, V_r), $$ $$W_l=L(x), W_r=R(y), W=(W_l, W_r), Z=(L(uvx), R(uvy))=(U_lV_l, V_rU_r),$$ and $$ Z'=(L(ux), R(uy))=(U_lW_l, W_rU_r).$$  By Theorem 5.2.1 in \cite{CNS1}, we have
$$\kappa_\chi(Z)=\sum_{\pi\in BNC(\chi)}\kappa_\pi(U)\kappa_{K_{BNC}(\pi)}(V)=\sum_{\pi\in BNC(\chi)}\kappa_\pi(U)\kappa_{K_{BNC}(\pi)}(W)=\kappa_\chi(Z'),$$ for every $\chi:\{1,..., n\}\rightarrow \{l,r\}$, where the second equation holds because,  by $(3)$, $(L(vx), R(vy))$ and $(L(x), R(y))$ are identically $*$-distributed.
Therefore, $(L(uvx), R(uvy))$ and $(L(ux), R(uy))$ are identically $*$-distributed.

By Theorem 14.4 in \cite{NS}, we have $$\varphi((uv)^n)=\varphi(uvuv\cdots uv)=\sum_{\pi\in NC(n)}\kappa_\pi(u)\varphi_{K_\pi}(v, ..., v)=0, n=1, 2, ...,$$ since $v$ is a Haar unitary, where $K\pi$ is the Kreweras complement of $\pi\in NC(n)$ (see Definition 9.21 in \cite{NS}).
Similarly, $\varphi((v^*u^*)^n)=\varphi(v^*u^*v^*u^*\cdots v^*u^*)=0, n\ge 1$. It follows that $uv$ is a Haar unitary.
By Theorem 3.2 in \cite{GK} and Theorem 3.2, $(L(uvx),R(uvy))$ (therefore, $(L(ux), R(uy))$) satisfies $(3.6)$.    Moreover,  $(3.7)$ holds, since $u$ is $*$-free from $\{x,y\}$. By the second part of the proof of Theorem 3.2, $(L(x), R(y))$ and $(L(ux), R(uy))$ are identically $*$-distributed.
\end{proof}

\section{ $\eta$-diagonal Pairs of Random Variables}

In this section, we study $\eta$-diagonal pairs of random variables, characterizing  $\eta$-diagonal pairs in terms of the $*$-moments of  the random variables.

\begin{Definition}[\cite{GS}] Let $n\in \mathbb{N}$ and $\chi:\{1, 2, \cdots, n\}\rightarrow \{l,r\}$. 
A partition $\pi\in BNC(\chi)$ is said to be bi-interval if every block of $\pi$ is a $\chi$-interval, that is, every block of the partition is an interval of natural numbers with respect to the new order $\prec_\chi$ defined at the beginning of Section 2. The set of all bi-interval partitions is denoted by $BI(\chi)$. Let $(\A,\varphi)$ be a non-commutative probability space. The $B$-$(l,r)$- cumulants are the multilinear functionals $B\chi:\A^n\rightarrow \mathbb{C}$ for $\chi:\{1, 2, \cdots, n\}\rightarrow \{l,r\}$ defined by the requirement $$\varphi(a_1,a_2, \cdots, a_n)=\sum_{\pi\in BI(\chi)}B_\pi(a_1, \cdots, n),\eqno (4.1)$$ where $$B_\pi(a_1,...,a_n)=\prod_{V\in\pi}B_{\chi|_V}((a_1,...,a_n)|_V).$$
A family $\{(\A_{k,l}, \A_{k,r}): k\in K\}$  of pairs of non-unital subalgebras in $(\A, \varphi)$ is said to be bi-Boolean independent if for all $n\in \mathbb{N}$, $\chi:\{1, 2, \cdots, n\}\rightarrow \{l,r\}$, $\omega:\{1,..., n\}\rightarrow K$, and  $a_j\in \A_{\omega(j), \chi(j)}$, $j=1, ..., n$, we have $$\varphi(a_1\cdots a_n)=\varphi_{\pi_{\omega, \chi}}(a_1,...,a_n),$$ where $\pi_{\omega,\chi}=\max\{\pi\in BI(\chi):\pi\le \omega^{-1}\}$, where $\omega^{-1}$ is the partition of $\{1,..., n\}$ defined by $i\sim_{\omega^{-1}}j$ if and only if $\omega(i)=\omega(j)$.
\end{Definition}

We give a straightforward proof of the following characterization of bi-free Boolean independence in terms of B-$(l,r)$-cumulants, without using either c-bi-free independence or the incidence algebra of $BI$.
\begin{Proposition}[Theorem 3.7 of \cite{GS}]A family $\{(\A_{k,l}, \A_{k,r})\}_{k\in K}$  of pairs of non-unital algebras in $(\A, \varphi)$ is  bi-Boolean independent if and only if  for every $n\ge 2$, $\chi:\{1, 2, \cdots, n\}\rightarrow \{l,r\}$, $\omega:\{1,..., n\}\rightarrow K$, and $a_j\in \A_{\omega(j), \chi(j)}$, $j=1, ..., n$, we have $$B_\chi(a_1, ..., a_n)=0,$$ whenever $\omega$ is not constant.
\end{Proposition}
\begin{proof} If $\{(\A_{k,l}, \A_{k,r})\}_{k\in K}$ is bi-free Boolean independent, then for every $n\ge 2$, $\chi:\{1, 2, \cdots, n\}\rightarrow \{l,r\}$, $\omega:\{1,..., n\}\rightarrow K$, and  $a_j\in \A_{\omega(j), \chi(j)}$, $j=1,..., n$, we have $$\varphi(a_1\cdots a_n)=\varphi_{\pi_{\omega, \chi}}(a_1,..., a_n)=\sum_{\pi\in BI(\chi), \pi\le \pi_{\omega, \chi}}B_{\pi}(a_1,..., a_n).$$ If $n=2$ and $\omega(1)\ne \omega(2)$, then $\varphi(a_1a_2)=\varphi_{\pi_{\omega,\chi}}(a_1,a_2)=\varphi(a_1)\varphi(a_2)=\sum_{\pi\in BI(\chi)}B_\pi(a_1,a_2)=\varphi(a_1)\varphi(a_2)+B_\chi(a_1, a_2)$. It implies that $B_\chi(a_1, a_2)=0$. Suppose that $B_\chi(a_1, ..., a_n)=0$, for $n\ge 2$, and $\omega$ is not constant. Now consider $\chi:\{1, ..., n+1\}\rightarrow \{l,r\}$ and $\omega:\{1, ..., n+1\}\rightarrow K$ is not constant. By the above definition,
\begin{align*}
B_\chi(a_1,..., a_{n+1})=&\varphi(a_1\cdots a_{n+1})-\sum_{\pi\in BI(\chi), \pi\ne 1_\chi}B_\pi(a_1,...,a_{n+1})\\
=&\sum_{\pi\in BI(\chi), \pi\le \pi_{\omega, \chi}}B_{\pi}(a_1,..., a_{n+1})\\
-&(\sum_{\pi\in BI(\chi), \pi\le \pi_{\omega, \chi}}B_{\pi}(a_1,..., a_{n+1})+\sum_{\pi\in BI(\chi), \pi\ne 1_\chi, \pi\nleq \pi_{\omega, \chi}}B_{\pi}(a_1, ..., a_{n+1}))\\
=& -\sum_{\pi\in BI(\chi), \pi\ne 1_\chi, \pi\nleq \pi_{\omega, \chi}}\prod_{V\in\pi}B_{\chi|V}((a_1, ..., a_{n+1})|V)=0,
\end{align*}
the last equality holds true because of the inductive hypothesis and the fact that, for each $\pi$ in the set of partitions of the second summand,  there is a block $V\in \pi$ such that $\omega|V$ is not constant.

The above discussions also show that vanishing of mixed B-$(l,r)$-cumulants implies bi-free Boolean independence for a family of pairs of non-unital algebras.
\end{proof}
We adjust  some notations in \cite{BNNS} to the present case. Let $\mathcal{W}^+:=\bigsqcup_{n=1}^\infty\{l,l^*, r, r^*\}^n$ be the set of all non-empty words over the four-letter alphabet $\{l,l^*, r, r^*\}$. The algebra of polynomials in four non-commutative variables $Z_l$, $Z_l^*$, $Z_r$, and $Z_r^*$
is denoted by $\mathbb{C}\langle Z_l, Z_l^*, Z_r, Z_r^* \rangle$. For $w=w_1 \cdots w_n\in \mathcal{W}^+,$ we write $Z_w=Z_{w_1}\cdots Z_{w_n}\in \mathbb{C}\langle Z_l, Z_l^*, Z_r, Z_r^*\rangle$, where $$Z_k=\left\{\begin{matrix}Z_l, & \text{if } k=l\\
Z_l^*, &\text{if } k=l^*\\
Z_r,& \text{ if } k=r\\
Z_r^*, & \text{ if } k=r^*.\end{matrix}\right.$$ We also use $Z_w$ to denote the tuple $(Z_{w_1},..., Z_{w_n})$.
An algebraic $*$-distribution is a unital linear functional $$\mu:\mathbb{C}\langle Z_l, Z_l^*, Z_r, Z_r^*\rangle\rightarrow \mathbb{C}.$$ The collection of all algebraic $*$-distributions of $\mathbb{C}\langle Z_l, Z_l^*, Z_r, Z_r^*\rangle$ is denoted by $\mathcal{D}(l,r,*)$. The algebra of all formal power series in four non-commutative indeterminates $z_l, z_r, z_l^*$ and $z_r^*$ is denoted by $\mathbb{C}\langle\langle z_l,z_l^*,z_r, z_r^*\rangle\rangle$. The collection of all power series in $\mathbb{C}\langle\langle z_l, z_l^*, z_r, z_r^*\rangle\rangle$ with vanishing constant coefficient is denoted by $\mathbb{C}_0\langle\langle z_l, z_l^*, z_r, z_r^*\rangle\rangle$. In the following, $\mu(Z_w)=\mu(Z_{w_1}\cdots Z_{w_n})$, while $B_\chi(Z_w)=B_{\chi}(Z_{w_1}, ..., Z_{w_n})$, and $\kappa_\chi(Z_w)=\kappa_{\chi}(Z_{w_1},..., Z_{w_n})$.

Applying Definition 7.1 in \cite{GS}, and Definition and Remark 2.3 in \cite{BNNS} to our case, we get the following definition.
\begin{Definition}
For $\mu\in \mathcal{D}(l,r,*)$, we define the following formal power series.
\begin{enumerate}
\item The moment series of $\mu$ is $M_\mu:=\sum_{w\in \mathcal{W}^+}\mu(Z_w)z_w\in \mathbb{C}_0\langle\langle z_l, z_l^*, z_r, z_r^*\rangle\rangle$.
\item The bi-Boolean $\eta$-series of $\mu$ is $$\eta_\mu=\sum_{w\in \mathcal{W}^+}B_{\chi_w}(Z_w)z_w,$$
\item The bi-free $R$-transform of $\mu$ is $R_\mu=\sum_{w\in \mathcal{W}^+}\kappa_{\chi_w}(Z_w)z_w$,
\end{enumerate}
where $\chi_w: \{1, 2, ..., n\}\rightarrow \{l,r\}$ is defined by $\chi_w(k)=l$, if $w_k\in\{l, l^*\}$; $\chi_w(k)=r$, if $w_k\in \{r, r^*\}$, for $w=w_1...w_n\in \mathcal{W}^+$,
$B_{\chi_w}$ is defined by (4.1).
\end{Definition}

\begin{Definition}
We say $w=w_1...w_n\in \W^+$ is  alternating, if $w_{\chi_w}:=(w_{s_{\chi_w}(1)}... w_{s_{\chi_w}(n)})\in W_1\cup W_2$, where $\chi_w:\{1,..., n\}\rightarrow \{l, r\}$ is the function defined in definition 4.3, and
\begin{enumerate}
\item   $ W_{1}:=\{l^{\omega(1)},  ..., l^{\omega(k)}, r^{\omega(k+1)},..., r^{\omega(2m)})\in \mathcal{W}^+:\omega:\{1,..., 2m\}\rightarrow \{1,*\},\omega(1)=1, \omega(i)\ne\omega(i+1), i=1,..., 2m-1, 0\le k\le 2m, m=1, 2, ... \}$,
\item  $ W_{2}:=\{l^{\omega(1)}, l^{\omega(2)},  ..., l^{\omega(k)}, r^{\omega(k+1)},..., r^{\omega(2m)})\in \mathcal{W}^+:\omega:\{1,..., 2m\}\rightarrow \{1,*\},\omega(1)=*,  \omega(i)\ne\omega(i+1), i=1,..., 2m-1, 0\le k\le 2m, m=1, 2, ... \}$.
\end{enumerate}
A word $w\in \mathcal{W}^+$ is said to be mixed alternating, if  $w_{\chi_w}=w_1\cdots w_p$, where $w_i\in W_{j_i}$, $j_1\ne j_2\ne \cdots\ne j_p$, and $j_1, j_2, ..., j_p\in \{1,2\}$.
\end{Definition}
\begin{Definition} A $*$-distribution $\mu\in \mathcal{D}(l,r,*)$ is said to be $\eta$-diagonal if $B_{\chi_\omega}(Z_w)=0$ whenever $w\in \mathcal{W}^+$ is not alternating. In this case, the $\eta$-series of $\mu$ has the following form $$\eta_\mu=\sum_{w_{\chi_w}\in W_1}B_{\chi_\omega}(Z_w)z_w+\sum_{w_{\chi_w}\in W_2}B_{\chi_\omega}(Z_w)z_w.$$

\end{Definition}

We shall give a characterization of $\eta$-diagonal distributions in $\mathcal{D}(l,r,*)$ in terms of their $*$-moments, similar to Theorem 2.8 in \cite{BNNS} for $\eta$-diagonal distributions of single random variables. We first give a couple of preliminary results.

\begin{Lemma} For $\mu\in \mathcal{D}(l,r,*)$, if $\mu(Z_w)=0$, whenever $w$ is not mixed alternating, then $B_{\chi_w}(Z_w)=0$, whenever $w$ is not mixed-alternating.
\end{Lemma}
\begin{proof} Let $w=w_1...w_n\in \mathcal{W}^+$ and $\chi_w:\{1, ..., n\}\rightarrow \{l,r\}$ defined in Definition 4.3. By Definition 4.1, $\mu(Z_w)=\sum_{\pi\in BI(\chi_w)}B_{\chi_w}(Z_w)$. By 3.3 in \cite{GS},  the Mobius function $$\mu_{BI}: BI:=\cup_{n\ge 1}\cup_{\chi:\{1, ..., n\}\rightarrow \{l,r\}}BI(\chi)\rightarrow \mathbb{Z}$$ is defined recursively by the equation $$\sum_{\tau\in BI(\chi), \sigma\le \tau\le \pi}\mu_{BI}(\tau, \pi)=\sum_{\tau\in BI(\chi), \sigma\le \tau\le \pi}\mu_{BI}(\sigma,\tau)=\left\{\begin{matrix}1, & \text{ if } \sigma=\pi \\
0, & \text{ otherwise}\end{matrix}.\right.$$ Then we have $$B_{\chi_w}(Z_w)=\sum_{\pi\in BI(\chi_w)}\mu_{\pi}(Z_w)\mu_{BI}(\pi, 1_{|w|}),$$ for $w\in \mathcal{W}^+$. The conclusion follows now from the fact that $w$ is mixed alternating if $w|_V$ is mixed alternating for every $V\in \pi\in BI(\chi)$.
\end{proof}
\begin{Lemma} Let $\mu, \nu\in \mathcal{D}(l,r,*)$. If  $\mu$ and $\nu$ satisfy the following conditions
\begin{enumerate}
\item  $\mu(Z_w)=\nu(Z_w)=0$, if $w\in \mathcal{W}^+$ is not mixed alternating;
\item for a mixed alternating word $w\in \mathcal{W}^+$ with canonical factorization $w_{\chi_w}=w_1\cdots w_d$, and $\pi_w=\{J_1, ..., J_d\}\in BI(\chi_w)$ such that $w_{\chi_w|_{J_i}}=w_i$, $i=1, 2, ..., d$, we have  $$\mu(Z_w)=\mu(Z_{w|_{J_1}})\cdots \mu(Z_{w|_{J_d}}), \nu(Z_w)=\nu(Z_{w|_{J_1}})\cdots \nu(Z_{w|_{J_d}});$$
\item also, $B^{(\mu)}_{\chi_w}(Z_w)=B^{(\nu)}_{\chi_w}(Z_w)$ for every alternating word $w\in \mathcal{W}^+$, where $B^{(\mu)}$ and $B^{(\nu)}$ are Boolean cumulant functions of $\mu$ and $\nu$, respectively,
\end{enumerate}
then $\mu=\nu$.
\end{Lemma}
\begin{proof}By conditions (1) and (2), it is sufficient to prove $\mu(Z_w)=\nu(Z_w)$ for an alternating word $w\in \mathcal{W}^+$. We prove the equality $\mu(Z_w)=\nu(Z_w)$ for $w_{\chi_w}\in W_1$. The proof for  the other case is essentially the same as  this case. By $(4.1)$, we have
\begin{align*}
\mu(Z_w)=&\sum_{\pi\in BI(\chi_w)}B^{(\mu)}_{\pi}(Z_w)\\
=&\sum_{\pi\in BI(\chi_w), |V|\in 2\mathbb{N}, \forall V\in \pi}\prod_{V\in \pi}B^{(\mu)}_{ \chi_w|_V}(Z_{w|_V})\\
=&\sum_{\pi\in BI(\chi_w), |V|\in 2\mathbb{N}, \forall V\in \pi}\prod_{V\in \pi}B^{(\nu)}_{ \chi_w|_V}(Z_{w|_V})=\nu(Z_w),
\end{align*} where the second and last equalities hold because $B_{\chi_w|_V}(Z_{w|_V})=0$  by Lemma 4.6, for $V\in \pi\in BI(\chi_w)$ such that $|V|$ is odd, since $w|_V$ is not mixed alternating; the third equality holds, because Condition $(3)$ and the fact that $w|_V$ is an alternating word in $W_1$, for a blck $V\in \pi$ if every block of $\pi$ is of an even set.
\end{proof}
\begin{Theorem} A $*$-distribution $\mu\in \mathcal{D}(l,r,*)$ is $\eta$-diagonal if and only if $\mu(Z_w)=0$, whenever $w\in \mathcal{W}^+$ is not mixed alternating, and $\mu(Z_w)=\mu(Z_{w|_{J_1}})\cdots \mu(Z_{w|_{J_d}})$ for every mixed alternating word $w$ with canonical factorization $w_{\chi_w}=w_1\cdots w_d$, where $\pi_w:=\{J_1, ..., J_d\}\in BI(\chi_{w})$ such that $w_{\chi_w}|_{J_i}=w_i$, $i=1, ..., p$.
\end{Theorem}
\begin{proof}
If a distribution $\mu$ is $\eta$-diagonal, then, for a non-mixed-alternating word $w\in \mathcal{W}^+$,  $$\mu(Z_w)=\sum_{\pi\in BI(\chi_w)}\prod_{V\in \pi}B_{\chi_w|_V}(Z_{w|_V})=0,$$ since there is at least one block $V\in \pi$ such that $w|_V$ is not alternating for every $\pi\in BI(\chi_w)$. Moreover, for a mixed alternating word $w\in \mathcal{W}^+$ with canonical factorization $w_{\chi_w}=w_1\cdots w_d$, let $\pi_w=\{J_1, ..., J_d\}\in BI(\chi_w)$ such that $w_{\chi_w}|_{J_i}=w_i$, for $i=1, 2, ..., d$. Then $B_{\rho, \chi_w}(Z_w)=0$ for $\rho\in BI(\chi_w)$, if $\rho\nleq \pi_w$, since $\mu$ is $\eta$-diagonal. It follows that
\begin{align*}\mu(Z_w)=&\sum_{\pi\in BI(\chi_w)}B_{\pi}(Z_w)=\sum_{\pi\in BI(\chi_w),\pi\le \pi_w}B_{\pi}(Z_w)\\
=&\sum_{\pi=\pi_1\cup\cdots \cup\pi_d\in BI(\chi_w), \pi_i\in BI(\chi_w|_{J_i}), i=1, ..., d}\prod_{i=1}^dB_{\pi_i}(Z_{w|_{J_i}})\\
=&\prod_{i=1}^d\sum_{\pi_i\in BI(\chi_w |_{J_i})}B_{\pi_i }(Z_{w|_{J_i}})=\mu(Z_{w|_{J_1}})\cdots \mu(Z_{w|_{J_d}}).
\end{align*}
Conversely, if $\mu\in \mathcal{W}^+$ satisfies the two conditions in this theorem, we define a $*$-distribution $\nu$ by assigning its Boolean cumulants $B^{(\nu)}_{\chi_w}(Z_w)=B^{(\mu)}_{\chi_w}(Z_w)$, for an alternating word $w\in \mathcal{W}^+$; and $B^{(\nu)}_{\chi_w}(Z_w)=0$, for a non-alternating word $w\in \mathcal{W}^+$, where $B^{(\mu)}$ is the Boolean cumulant function of $\mu$. By the definition of $\eta$-diagonal distributions, $\nu$ is $\eta$-diagonal. By the proof above, $\nu(Z_w)=0$, if $w$ is not mixed alternating, and $\nu(Z_w)=\nu(Z_{w_1})\cdots \nu(Z_{w_d})$ for a mixed alternating word $w=w_1...w_d$. By Lemma 4.7, $\mu=\nu$ is $\eta$-diagonal.
\end{proof}
For a pair $(a,b)$ of random variables in a non-commutative probability space $(\A, \varphi)$, the bi-Boolean $\eta$ series is $$\eta_{a,b}(z_l, z_r)=\sum_{n=1}^\infty \sum_{\chi:\{1,..., n\}\rightarrow \{l,r\}}B_\chi(Z_1, ..., Z_n)z_{\chi(1)}\cdots z_{\chi(n)}, $$ where $Z_k=a$ if $\chi(k)=l$; $Z_k=b$ if $\chi(k)=r$. A pair $(a,b)$ is $\eta$-diagonal, if $B_{\chi}(Z_1,.., Z_n)=0$ unless $n$ is even and $(Z_{s_{\chi}(1)}, ..., Z_{s_\chi(n)})=(a,b,a, b,..., a, b)$ or $Z_{s_{\chi}(1)}, ..., Z_{s_\chi(n)})=(b,a,b, a,..., b, a)$.
\begin{Proposition} Let $(x,y)$ be a $\eta$-diagonal pair of random variables in a $*$-probability space $(\A, \varphi)$. Then we have $$\eta_{(xx^*, yy^*)}(z_l, z_r)=\sum_{n=1}^\infty(B_{2n}(x, x^*,..., x, x^*)z_l^n+B_{2n}(y^*, y, ..., y^*, y)z_r^n),$$
$$\eta_{(x^*x, y^*y)}(z_l, z_r)=\sum_{n=1}^\infty(B_{2n}(x^*, x,..., x^*, x)z_l^n+B_{2n}(y, y^*, ..., y, y^*)z_r^n). $$
\end{Proposition}
\begin{proof}We prove the first formula only. The proof for the other is essentially the same. For $n\in \mathbb{N}$ and $\chi:\{1,..., n\}\rightarrow \{l,r\}$,
let $Z_l=xx^*$, $Z_r=yy^*$, and $Z_\chi=Z_{\chi(1)},..., Z_{\chi(n)}$ in $B_\chi(Z_\chi)$ and $\varphi_\pi(Z_\chi)$, and $Z_\chi=Z_{\chi(1)}\cdots Z_{\chi(n)}$ in $\varphi(Z_\chi)$.
By 3.3 in \cite{GS}, we have
$$B_\chi(Z_\chi)=\sum_{\pi\in BI(\chi)}\varphi_\pi(Z_\chi)\mu_{BI}(\pi,1_n)=\sum_{\pi\in BI(\chi)}(-1)^{|\pi|-1}\varphi_\pi(Z_\chi).$$
Let $\hat{\chi}:\{1,..., 2n\}\rightarrow \{l,r\}$ by $\hat{\chi}(2k-1)=\hat{\chi}(2k)=\chi(k)$, for $k=1, ..., n$. For $\pi=\{V_1, ..., V_d\}\in BI(\chi)$,
where blocks are arranged in an increasing order with respect to $\prec_\chi$, that is,
 $s_\chi(1)=\min_{\prec_\chi}(V_1)$, and $\max_{\prec_\chi}(V_i)\prec_\chi\min_{\prec_\chi}(V_{i+1})$, for $i=1, ..., n-1$,
 define a partition $\hat{\pi}=\{\hat{V}_1,..., \hat{V}_d\}\in BI(\hat{\chi})$, where $2k-1, 2k\in \hat{V}_i\in \hat{\pi}$ if and only if $k\in V_i\in \pi$, for $i=1, ..., d$.
 Then we have $$\varphi_\pi(Z_\chi)\mu_{BI}(\pi, 1_n)=(-1)^{|\hat{\pi}|-1}\varphi_{\hat{\pi}}(Y_{\hat{\chi}})=\varphi_{\hat{\pi}}(Y_{\hat{\chi}})\mu_{BI}(\hat{\pi},1_{2n}),$$
 where $Y_{\hat{\chi}}=Y_{\hat{\chi}(1)}\cdots Y_{\hat{\chi}(2n)}$, $Y_{\hat{\chi}(2k-1)}=x, Y_{\hat{\chi}(2k)}=x^*$, if $\chi(k)=l$; $Y_{\hat{\chi}(2k-1)}=y^*, Y_{\hat{\chi}(2k)}=y$, if $\chi(k)=r$, for $k=1,...,n$.

 Let $\rho\in BI(\hat{\chi})$, and $\rho=\{S_1, S_2,..., S_d \}$ such that, for each block $S_i$, $|S_i|$ is even. Let $$S_i=\{s_{\hat{\chi}}(2p_i+1), ...,s_{\hat{\chi}}(2p_i+2k_i), s_{\hat{\chi}}(2p_i+2k_i+1), ..., s_{\hat{\chi}}(2p_i+2q_i)\},$$ where $\hat{\chi}(s_{\hat{\chi}}(2p_i+1))=...=\hat{\chi}(s_{\hat{\chi}}(2p_i+2k_i))=l$, and $\hat{\chi}(s_{\hat{\chi}}(2p_i+2k_i+1))=...=\hat{\chi}(s_{\hat{\chi}}(2p_i+2q_i))=r$.  Then $S_i=\hat{V}_i$, where $V_i=\{s_\chi(p_i+1),..., s_\chi(p_i+q_i)\}$, for $i=1, 2, ..., d$. Let $\pi=\{V_1, ..., V_d\}$. We then have $\pi\in BI(\chi)$ and $\rho=\hat{\pi}$. It implies that for $\rho\in BI(\hat{\chi})$, $\rho=\hat{\pi}$ for some $\pi\in BI(\chi)$ if and only if, for each $S\in \rho$, $|S|$ is even. Therefore, if $\rho\in BI(\hat{\chi})$ does not have a form $\hat{\pi}$ for some $\pi\in BI(\chi)$, there exists a block $S\in \rho$ such that $|S|$ is odd. By Theorem 4.8, $\varphi(Y_{\hat{\chi}|_{S}})=0$. It implies that
\begin{align*}
B_\chi(Z_\chi)=&\sum_{\pi\in BI(\chi)}(-1)^{|\pi|-1}\varphi_\pi(Z_\chi)=\sum_{\pi\in BI(\chi)}(-1)^{|\hat{\pi}|-1}\varphi_{\hat{\pi}}(Y_{\hat{\chi}})\\
=&\sum_{\rho\in BI(\hat{\chi})}(-1)^{|\rho|-1}\varphi_\rho(Y_{\hat{\chi}})=B_{\hat{\chi}}(Y_{\hat{\chi}})\\
=&\left\{\begin{array}{ll}B_{2n}(x,  x^*, ..., x, x^* ), & \text{ if } \chi\equiv l,\\
B_{2n}(y^*, y, ..., y^*, y), & \text{ if } \chi\equiv r,\\
0, & \text{ if } \chi \text{ is not constant},\end{array}\right.
\end{align*}  where $B_{\hat{\chi}}(Y_{\hat{\chi}})=0$, if $\chi$ is not constant, since $(Y_{\hat{\chi}(s_{\hat{\chi}}(1))}, ..., Y_{\hat{\chi}(s_{\hat{\chi}}(2n))})=(x, x^*, ..., x, x^*, y^*, y, ..., y^*, y)$ is not alternating.
 By the definition of $\eta$-series for $(xx^*,yy^*)$, we have
$$\eta_{(xx^*, yy^*)}=\sum_{n=1}^\infty(B_{2n}(x, x^*,..., x, x^*)z_l^n+B_{2n}(y^*, y, ..., y^*, y)z_r^n). $$
\end{proof}
\begin{Corollary} There is an $\eta$-diagonal pair $(x,y)$ of random variables in a $*$-probability space $(\A, \varphi)$ such that $(xx^*, yy^*)$ and $(x^*x, y^*y)$ are not bi-Boolean independent.
\end{Corollary}
\begin{proof}
By Definition 7.1 in \cite{GS}, $M, R, \eta: \mathcal{D}(l, r, *)\rightarrow \mathbb{C}_0\langle\langle z_l, z_l^*, z_r, z_r^*\rangle\rangle$ are bijections.
Therefore, we can define a distribution $\mu\in \D(l,r,*)$ by the equation $\eta_\mu=z_lz_l^*z_r^*z_r$, that is, there is a pair  $(x,y)$ of random variables in a $*$-probability space $(\A, \varphi)$ such that $B_\chi(x, x^*, y^*, y)=1$, where $\chi:(1, 2, 3, 4)\mapsto (l,l, r,r)$, while all other bi-Boolean cumulants of $(x,y)$ vanish. Then $(x,y)$ is $\eta$-diagonal. It implies by the proof of Proposition 4.9 that
$$B_\chi(xx^*, y^*y)=B_{\hat{\chi}}(x, x^*, y^*, y)=1,$$ where $\chi(1)=l, \chi(2)=r$, $\hat{\chi}(1)=\hat{\chi}(2)=l$, and $\hat{\chi}(3)=\hat{\chi}(4)=r$. By Proposition 4.2, $(xx^*, yy^*)$ and $(x^*x, y^*y)$ are not bi-Boolean independent.
\end{proof}

\section{Boolean Independence with Amalgamation}
It was proved in Theorem 1.2 in \cite{NSS} that an element $x$ in a $*$-probability space $(\A, \varphi)$ is $R$-diagonal if and only if the unital algebra $\Z$ generated by $\left(\begin{matrix}0&x\\
x^*&0\end{matrix} \right)$ and the scalar diagonal matrix algebra $\D_2$ is free from  $M_2(\mathbb{C})$  with amalgamation over $\D_2$ in $(M_2(\A), \D_2, F_2)$, where $F_2:M_2(\A)\rightarrow D_2$, $F_2([a_{ij}])= \left(\begin{matrix}\varphi(a_{11})&0\\
0& \varphi(a_{22})\end{matrix}\right)$, for $[a_{ij}]\in M_2(\A)$. In this section we study Boolean independence of the above two algebras with amalgamation over $\D_2$.

Let's recall some basic facts on Boolean independence from \cite{GS} and \cite{MP}. Let $\A_1$ and $\A_2$ be two subalgebras of a non-commutative probability space $(\A, \varphi)$. We say $\A_1$ and $\A_2$ are Boolean independent if $$\varphi (Z_1Z_2\cdots Z_n)=\varphi(Z_1)\varphi(Z_2)\cdots\varphi(Z_n),\eqno (5.1)$$ for $Z_i \in \A_{j_i}, j_i\in \{1,2\}, i=1, 2, \cdots, n$, $j_1\ne j_2\ne\cdots \ne j_n$.  Let $\B$ be a subalgebra of $\A$.
A subalgebra $\A_1$ of $\A$ is called a $\B$-subalgebra if $\B\subseteq \A_1$ or $\B\bigsqcup \A_1$ is an algebra. Two $\B$-subalgebras $\A_1$ and $\A_2$ are Boolean independent over $\B$ if $(5.1)$ holds. By Remark 4.2 in \cite{MP}, if $\A_1$ and $\A_2$ are Boolean independent over $\B$, then $$E (Z_1\cdots Z_n)=E(Z_1)\cdots E(Z_n),\eqno (5.2)$$ for $Z_i \in \A_{j_i}, j_i\in \{1,2\}, i=1, 2, \cdots, n$, $j_1\ne j_2\ne\cdots \ne j_n$, where $E:\A\rightarrow \B$ is the conditional expectation (see Sections 2 and 4 in \cite{MP}).

 If $\A_2$ is unital, then we have $\varphi(A^n)=\varphi(A1A1\cdots A1)=\varphi(A)\varphi(1)\cdots \varphi(A)\varphi(1)=\varphi(A)^n$, for all $A\in \A_1$. To avoid this trivial case, Gu and Skoufranis \cite{GS} studied (bi-)Boolean independence for only non-unital (pairs of) subalgebras. In Definition 8.3 in \cite{GS}, a $\B$-subalgebra $\C$ of $\A$ was defined as a subalgebra satisfying  $\varepsilon ( \B\otimes 1_{\B})\subseteq \C$, where $\varepsilon: \B\otimes \B^{op}\rightarrow \A$ is a unital homomorphism such that $\varepsilon|_{\B\otimes 1_B}$ is injective, and $\B$ is a unital algebra over $\mathbb{C}$. It follows that $\varepsilon (1_B\otimes 1_B)=1_\A\in \C$, that is, $\C$ is a unital subalgebra. But Definition 8.3 in \cite{GS} defines bi-Boolean independence of non-unital $\B$-subalgebra pairs, which leads a contradiction.

  We provide another way to avoid the trivial case that $\varphi(A^n)=\varphi(A)^n$ in defining Boolean independence of subalgebras.  The new definition is equivalent to the well-known definition of Boolean independence for non-unital subalgebras, and avoids the contradiction in Definition 8.3 in \cite{GS}. We consider Boolean independence instead of bi-Boolean independence.

 \begin{Definition} Two subalgebras $\A_1$ and $\A_2$ of $(\A, \varphi)$ are Boolean independent if $(5.1)$ holds for $Z_i \in \A_{j_i}\setminus \mathbb{C}1, j_i\in \{1,2\}, i=1, 2, \cdots, n$, $j_1\ne j_2\ne\cdots \ne j_n$.

 Let $\B$ be a subalgebra of $(\A, \varphi)$. Two $\B$-subalgebras $\A_1$ and $\A_2$ of $(\A, \varphi)$ are Boolean independent over $\B$ if $(5.2)$ holds for $Z_i \in \A_{j_i}\setminus \mathbb{C}1, j_i\in \{1,2\}, i=1, 2, \cdots, n$, $j_1\ne j_2\ne\cdots \ne j_n$.
 \end{Definition}
 \begin{Theorem}
Let $(\A,\varphi)$ be a $*$-probability space. If the  unital algebras $\mathcal{Z}$   and $M_2(\mathbb{C})$  are Boolean independent with amalgamation over  $\D_2$ in $(M_2(\A), \D_2, F_2)$, then $$\varphi((xx^*)^n)=\varphi((x^*x)^n)=\varphi(x^n)=\varphi((x^*)^n)=0,  n\in \mathbb{N}.$$ If every non-zero random variable $x\in \A$ has a non-zero distribution, then the two algebras $\Z$ and $M_2(\mathbb{C})$ can never be Boolean independent over $\D_2$ for a non-zero $x$.
\end{Theorem}
\begin{proof}
 Let $I_2$ be the unit of the algebra $M_2(\A)$. By \cite{NSS},  every matrix  $M\in \mathcal{Z}\setminus \mathbb{C}I_2$ has the form $$M=\left(\begin{matrix}\alpha_{11}(xx^*)^{m_1}&\alpha_{12}x(x^*x)^{m_2}\\
\alpha_{21}x^*(xx^*)^{m_3}& \alpha_{22}(x^*x)^{m_4}\end{matrix}\right),
m_1, m_2, m_3, m_4\in \mathbb{N}\cup\{0\},
 \alpha_{ij}\in \mathbb{C}, i, j=1,2, $$ for $M$ is not in $\D_2$, or
 $$M=\left(\begin{matrix}\alpha_{11}&0\\
0& \alpha_{22}\end{matrix}\right): \alpha_{11}\ne \alpha_{22}, \alpha_{11}, \alpha_{22} \in \mathbb{C}.$$
 The non-scalar part   $M_2(\mathbb{C})\setminus \mathbb{C}I_2$ is equal to  $$\left\{\left(\begin{matrix}\beta_{11}&\beta_{12}\\
\beta_{21}&\beta_{22}\end{matrix}\right): \beta_{11}\ne \beta_{22}, \text{ if } \beta_{12}=\beta_{21}=0,\ \  \beta_{ij}\in \mathbb{C}, i,j=1,2\right\}.$$

  Let $m_1, m_2, m_3\in \mathbb{N}\cup \{0\}$, and
  $$Z_1=\left(\begin{matrix}(xx^*)^{m_1}&0\\
  0&1+(xx^*)^{m_1}\end{matrix}\right), Z_2=\left(\begin{matrix}0&x(x^*x)^{m_2}\\
  0&1\end{matrix}\right), Z_3=\left(\begin{matrix}0&0\\
  x^*(xx^*)^{m_3}&1\end{matrix}\right)\in \mathcal{Z}\setminus \mathbb{C}I_2,$$
  $$ A_1=\left(\begin{matrix}1&0\\
  0&0\end{matrix}\right), A_2=\left(\begin{matrix}0&0\\
  0&1\end{matrix}\right)\in \M. $$ Then $$F_2(Z_1A_1Z_2A_2Z_3)=\left(\begin{matrix}\varphi((xx^*)^{m_1+m_2+m_3+1})&0\\
0& 0\end{matrix}\right).$$
If $\mathcal{Z}$ and $M_2(\mathbb{C})$ are Boolean independent with amalgamation over $\D_2$ in $(\L(M_2(\A)), \D_2, F_2)$, we have  $$F_2(Z_1A_1Z_2A_2Z_3)=F_2(Z_1)A_1F_2(Z_2)A_2F_2(Z_2)=0.$$ It implies  that $\varphi((xx^*)^n)=0$, for $n\in\mathbb{N}$, if $\Z$ and $M_2(\mathbb{C})$ are Boolean independent with amalgamation over $D_2$, since $m_1+m_2+m_3+1\ge 1$.
Very similarly, let $W=\left(\begin{matrix}1+(x^*x)^{m_1}&0\\
  0&(x^*x)^{m_1}\end{matrix}\right)\in \mathcal{Z}\setminus \mathbb{C}I_2$, for $m_1\in \mathbb{N}\cup \{0\}$. We then have $$0=F_2(WA_2Z_3A_1Z_2)=\left(\begin{matrix}0&0\\
0&\varphi((x^*x)^{m_1+m_2+m_3+1})\end{matrix}\right), $$ if $\mathcal{Z}$ and $M_2(\mathbb{C})$ are Boolean independent with amalgamation over $\D_2$ in $(\L(M_2(\A)), \D_2, F_2)$. It follows that $\varphi((x^*x)^n)=0, n\in \mathbb{N}$.

If, furthermore, every non-zero element in $(\A, \varphi)$ has a non-zero distribution, we get $x^*x=xx^*=0$.  If $x\ne 0$, let $A=\left(\begin{matrix}0&x\\
x^*& 0\end{matrix}\right)\in \Z\setminus \mathbb{C}I_2$, and $B=\left(\begin{matrix}0&1\\
1& 0\end{matrix}\right)\in M_2(\mathbb{C})\setminus \mathbb{C}I_2$. We have $$F_2((AB)^n)=\left(\begin{matrix}\varphi((x^n)&0\\
0& \varphi((x^*)^n)\end{matrix}\right)=F_2(A)F_2(B)\cdots F_1(A)F_2(B)=0,n=1, 2, ....$$  It implies that  all $*$-moments of $x$ are zeros,therefore, $x=0$, which leads a contradiction.

  Hence, $\Z$ and $M_2(\mathbb{C})$ are not Boolean independent over $\D_2$, if $x\ne 0$.

\end{proof}

\end{document}